\documentclass{amsart}

\usepackage{amsmath,
            amssymb,
            bibentry,
            enumitem,
            graphicx,
            mathrsfs,
            tabu,
            tikz,
            xypic}
\usepackage[comma,numbers,sort,square]{natbib}
\usepackage[colorlinks=true,
            linktocpage=true,
            linkcolor=magenta,
            citecolor=magenta,
            urlcolor=magenta]{hyperref}
\PassOptionsToPackage{hyphens}{url}
\usetikzlibrary{arrows,shapes,snakes,positioning}
\usepackage{pstricks,pst-node,pst-tree}

\definecolor{grey}{rgb}{0.95,0.95,0.95}
\definecolor{green}{rgb}{0.2,0.6,0.4}

\newtheorem{theorem}{Theorem}[section]

\newtheorem{lemma}[theorem]{Lemma}
\newtheorem{corollary}[theorem]{Corollary}
\theoremstyle{definition}
\newtheorem{definition}[theorem]{Definition}
\newtheorem{statement}[theorem]{Statement}

\newtheorem{remark}[theorem]{Remark}

\newtheoremstyle{principle}{}{}{\itshape}{}{\bfseries}{.}{.5em}{\thmnote{#3}#1}
\theoremstyle{principle}

\newcommand{\Nb}{\mathbb{N}}

\newcommand{\Tb}{\mathbb{T}}

\newcommand{\Psf}{\mathsf{P}}

\newcommand{\Ccal}{\mathcal{C}}

\newcommand{\Fcal}{\mathcal{F}}
\newcommand{\Gcal}{\mathcal{G}}

\newcommand{\Ical}{\mathcal{I}}

\newcommand{\Lcal}{\mathcal{L}}
\newcommand{\Mcal}{\mathcal{M}}
\newcommand{\Pcal}{\mathcal{P}}

\newcommand{\Rcal}{\mathcal{R}}

\newcommand{\uh}{{\upharpoonright}}

\renewcommand{\setminus}{\smallsetminus}

\newcommand{\cond}[1]{\left\{\begin{array}{ll} #1 \end{array}\right.}

\def\qt#1{``#1''}%

\newcommand{\s}[1]{\ensuremath{\sf{#1}}}

\DeclareMathOperator{\rca}{\s{RCA}_0}
\DeclareMathOperator{\piooca}{\Pi^1_1\s{CA}}
\DeclareMathOperator{\aca}{\s{ACA}}
\DeclareMathOperator{\atr}{\s{ATR}}
\DeclareMathOperator{\wkl}{\s{WKL}}

\DeclareMathOperator{\dnc}{\s{DNC}}

\DeclareMathOperator{\rt}{\s{RT}}

\DeclareMathOperator{\ads}{\s{ADS}}
\DeclareMathOperator{\sads}{\s{SADS}}

\DeclareMathOperator{\cac}{\s{CAC}}

\DeclareMathOperator{\emo}{\s{EM}}
\DeclareMathOperator{\semo}{\s{SEM}}

\DeclareMathOperator{\amt}{\s{AMT}}

\DeclareMathOperator{\gap}{\s{GAP}}
\DeclareMathOperator{\dgap}{\s{DGAP}}

\usetikzlibrary{shapes,arrows}
\usetikzlibrary{decorations.markings}
\definecolor{lightblue}{HTML}{e6e6e6}
\definecolor{lightred}{HTML}{eca6a6}
\definecolor{lightgreen}{RGB}{164,244,140}

\newtheoremstyle{custom}% name of the style to be used
  {10pt}% measure of space to leave above the theorem. E.g.: 3pt
  {10pt}% measure of space to leave below the theorem. E.g.: 3pt
  {\normalfont}% name of font to use in the body of the theorem
  {}% measure of space to indent
  {\bfseries}% name of head font
  {}% punctuation between head and body
  { }% space after theorem head; " " = normal interword space
  {}% Manually specify

\theoremstyle{custom}

\usepackage{xcolor}	
\usepackage{soul}

\newcommand{\N}{\mathbb{N}}

\newcommand{\dom}{\operatorname{dom}}

\newcommand{\height}{\operatorname{ht}}
\newcommand{\Dec}{\mathrm{Dec}}
\newcommand{\Ib}{\mathbb{I}}

\begin{document}

\title{Thin set theorems and cone avoidance}

\author{Peter Cholak}
\address{Department of Mathematics\\
  University of Notre Dame\\
  Notre Dame, Indiana U.S.A.}  \email{Peter.Cholak.1@nd.edu}

\author{Ludovic Patey}
\address{Institut Camille Jordan\\
  Universit\'e Claude Bernard Lyon 1\\
  43 boulevard du 11 novembre 1918\\
  F-69622 Villeurbanne Cedex} \email{ludovic.patey@computability.fr}

\thanks{This work was partially supported by a grant from the Simons
  Foundation (\#315283 to Peter Cholak) and NSF Conference grants
  (DMS-1640836 and DMS-1822193 to Peter Cholak).}

\begin{abstract}  
  The thin set theorem $\rt^n_{<\infty,\ell}$ asserts the existence, for every
  $k$-coloring of the subsets of natural numbers of size $n$, of an
  infinite set of natural numbers, all of whose subsets of size $n$
  use at most $\ell$ colors.  Whenever $\ell = 1$, the statement
  corresponds to Ramsey's theorem. From a computational viewpoint, the
  thin set theorem admits a threshold phenomenon, in that whenever the
  number of colors $\ell$ is sufficiently large with respect to the
  size $n$ of the tuples, then the thin set theorem admits strong cone
  avoidance.%  In this paper, we exhibit the exact bound on this
  % threshold, given the size of the tuples, and give new insights on
  % the computational nature of Ramsey-type theorems.
  
  Let $d_0, d_1, \dots$ be the sequence of Catalan numbers.  For
  $n \geq 1$, $\rt^n_{<\infty, \ell}$ admits strong cone avoidance if
  and only if $\ell \geq d_n$ and cone avoidance if and only if
  $\ell \geq d_{n-1}$.  We say that a set $A$ is
  \emph{$\rt^n_{<\infty, \ell}$-encodable} if there is an instance of
  $\rt^n_{<\infty, \ell}$ such that every solution computes $A$.  The
  $\rt^n_{<\infty, \ell}$-encodable sets are precisely the
  hyperarithmetic sets if and only if $\ell < 2^{n-1}$, the arithmetic
  sets if and only if $2^{n-1} \leq \ell < d_n$, and the computable
  sets if and only if $d_n \leq \ell$.
\end{abstract}

\maketitle

\section{Introduction}\label{sect:introduction}

Ramsey's theorem asserts the existence, for every $k$-coloring of the
subsets of natural numbers of size $n$, of an infinite set of natural
numbers, all of whose subsets of size $n$ are monochromatic.  For
notation ease, as in many papers on Ramsey's theorem, we will consider
a set of $n$ natural numbers as an increasing $n$-tuple.  Ramsey's theorem plays a
central role in reverse mathematics, as Ramsey's theorem for pairs
historically provides an early example of a theorem which escapes the
main observation of the early reverse mathematics, namely, the \qt{Big
  Five} phenomenon~\cite{Simpson2009Subsystems}.  From a computational
viewpoint, Ramsey's theorem for $n$-tuples with $n \geq 3$ admits
computable $2$-colorings of the $n$-tuples such that every
monochromatic set computes the halting set~\cite{Jockusch1972Ramseys},
while Ramsey's theorem for pairs does not~\cite{Seetapun1995strength}.

More recently, Wang~\cite{Wang2014Some} considered a weakening of
Ramsey's theorem now known as the \emph{thin set theorem}, in which
the constraint of monochromaticity of the resulting set is relaxed so
that more colors are allowed. He proved that for every size $n$ of the
tuples, there exists a number of colors $\ell$ such that every
computable coloring of the $n$-tuples into finitely many colors, there
is an infinite set of natural numbers whose $n$-tuples have at most
$\ell$ colors and which does not compute the halting set.  On the
other hand, Dorais et al.~\cite{Dorais2016uniform} proved that
whenever the number of colors $\ell$ is not large enough with respect
to the size of the tuples $n$, then this is not the case. However, the
lower bound of Dorais et al.\ grows slower than the upper bound of
Wang. Therefore, Wang~\cite{Wang2014Some} naturally asked where the
threshold lies.

In this paper, we address this question by exhibiting the exact bound
at which this threshold phenomenon occurs, and reveal an intermediary
computational behavior of the thin set theorem whenever the number
$\ell$ of allowed colors in the outcome is exponential, but not large
enough. We also provide some insights about the nature of the
computational strength of Ramsey-type theorems.

\subsection{Reverse mathematics and Ramsey's theorem}

Reverse mathematics is a foundational program started by Harvey
Friedman, which seeks to determine the optimal axioms to prove
ordinary theorems. It uses the framework of second-order arithmetic,
with a base theory, $\rca$, informally capturing \emph{computable
  mathematics}. The early study of reverse mathematics revealed an
empirical structural phenomenon. More precisely, there are four
axioms, namely, $\wkl$ (weak K\"onig's lemma), $\aca$ (arithmetical
comprehension axiom), $\atr$ and $\piooca$, in increasing logical
order, such that almost every theorem of ordinary mathematics is
either provable in $\rca$ (hence computationally true), or equivalent
over $\rca$ to one of those four systems. These systems, together with
$\rca$, form the \qt{Big Five}~\cite{Montalban2011Open}. See
Simpson~\cite{Simpson2009Subsystems} for a reference book on the early
reverse mathematics.
 
Among the theorems studied in reverse mathematics, Ramsey's theorem
received a special attention, as it was historically the first example
of a theorem escaping this structural phenomenon. In what follows,
$[X]^n$ denotes the set of all unordered $n$-tuples over~$X$.
 
\begin{definition}[Ramsey's theorem]
  Given $n, k \geq 1$, $\rt^n_k$ is the statement \qt{For every
    coloring $f : [\omega]^n \to k$, there is an infinite
    \emph{$f$-homogeneous} set, that is, a set $H \subseteq \omega$
    such that $|f[H]^n| = 1$.}
\end{definition}

Ramsey's theorem can be seen as a mathematical problem, expressed in
terms of \emph{instances} and \emph{solutions}. Here, an instance is a
coloring $f : [\omega]^n \to k$, and a solution to $f$ is an infinite
$f$-homogeneous set.

Specker \cite{MR0278941} and Jockusch~\cite{Jockusch1972Ramseys}
studied Ramsey's theorem from a computational viewpoint.  Specker
\cite{MR0278941} showed that there is a $2$-coloring of pairs (of
natural numbers) $f$ with no infinite computable $f$-homogeneous set.
When formalized in the framework of reverse mathematics, this shows
that $\rt^2_2$ does not hold in $\rca$.  Jockusch constructed, for
every $n \geq 3$, a computable coloring $f : [\omega]^n \to 2$ such
that every infinite $f$-homogeneous set computes the halting set. When
formalized in the framework of reverse mathematics, this shows that
$\rt^n_2$ is equivalent to $\aca$ over $\rca$ whenever $n \geq 3$.

The case of Ramsey's theorem for pairs was a long-standing open
problem, until solved by Seetapun (see \cite{Seetapun1995strength})
using what is now known as \emph{cone avoidance}.

\begin{definition}
  A problem $\Psf$ admits \emph{cone avoidance} if for every pair of
  sets $C \not \leq_T Z$ and every $Z$-computable instance of $\Psf$,
  there is a solution $Y$ such that $C \not \leq_T Y \oplus Z$.
\end{definition}

Seetapun and Slaman proved that Ramsey's theorem for pairs ($\rt^2_k$)
admits cone avoidance. In particular, when taking $Z$ to be a
computable set and $C$ to be the halting set, this shows that every
computable instance of $\rt^2_k$ admits a solution which does not
compute the halting set. Proving that a problem admits cone avoidance
implies in particular that this problem does not imply~$\aca$
over~$\rca$.

\subsection{The thin set theorem}
Friedman~\cite{FriedmanFom53free} first suggested studying a weakening
of Ramsey's theorem, the thin set theorem, which asserts that every
coloring of $f : [\omega]^n \to \omega$ admits an infinite set
$H \subseteq \omega$ such that $[H]^n$ avoids at least one color, that
is, $f[H]^n \neq \omega$. We shall however consider stronger
statements defined by Miller~\cite{Montalban2011Open} and that we also
refer to as thin set theorems.

\begin{definition}[Thin set theorem]
  Given $n, \ell \geq 1$, $\rt^n_{<\infty,\ell}$ is the statement
  ``For every $k$ and every coloring $f : [\omega]^n \to k$, there is
  an infinite set $H$ such that $|f[H]^n| \leq \ell$.''
\end{definition}

$H$ is called $f,\ell$-\emph{thin}.  Most of the time we will drop the
$f$ and $\ell$ since they are understood by content.  More or less,
Friedman's thin set theorem restricted $\ell$ and $k$ to when $k=\ell+1$.
But, in the setting of reverse mathematics and computability theory, the number of colors $k$ of the instance $f : [\omega]^n \to k$ is not relevant, since 
we are allowed to apply the thin set theorem multiple times. (Apply the restricted thin set theorem
$k - \ell$ times to get the desired thin set.) Whenever $\ell = 1$, we
obtain Ramsey's theorem.

% % Whenever $k = \ell+1$, we obtain a statement close to the one
% % defined
% % by Friedman, but for finite colorings. However, in reverse
% % mathematics, the initial number of colors $k$ of
% the coloring $f : [\omega]^n \to k$ is not relevant as soon as
% $k \geq \ell$.

Wang~\cite{Wang2014Some} studied the thin set theorem (under the name
Achromatic Ramsey Theorem), and proved that whenever $\ell$ is
sufficiently large with respect to $n$, then $\rt^n_{<\infty, \ell}$
admits cone avoidance.  His proof is an inductive interplay between
the combinatorial and the computational weakness of the thin set
theorems, which involves the notion of \emph{strong cone avoidance}.

\begin{definition}
  A problem $\Psf$ admits \emph{strong cone avoidance} if for every
  pair of sets $C, Z$ with $C \not \leq_T Z$ and every (arbitrary)
  instance of $\Psf$, there is a solution $Y$ such that
  $C \not \leq_T Y \oplus Z$.
\end{definition}

Contrary to cone avoidance, strong cone avoidance does not consider
only $Z$-computable instances of $\Psf$, but arbitrary ones. Thus,
while cone avoidance expresses a \emph{computational weakness} of the
problem $\Psf$, strong cone avoidance reveals a \emph{combinatorial
  weakness}, in the sense that even with an unlimited amount of power
for defining the instance of $\Psf$, one cannot code the set $C$ in
its solutions.

Note that in the case of the thin set theorems, there is a formal
relationship between strong cone avoidance and cone avoidance.

\begin{theorem}\label{thm:bridge-strong-to-non-strong}
  For every $n, \ell \geq 1$, $\rt^n_{<\infty, \ell}$ admits strong
  cone avoidance if and only if $\rt^{n+1}_{<\infty, \ell}$ admits
  cone avoidance.
\end{theorem}
\begin{proof}
  $\Rightarrow$ (Wang~\cite{Wang2014Some}). Fix $C \not \leq_T Z$ and
  a $Z$-computable coloring $f : [\omega]^{n+1} \to k$. Let
  $G = \{x_0 < x_1 < \dots \}$ be a sufficiently generic set for
  computable Mathias genericity~\cite{Cholak2014Generics}. In
  particular, for every $\vec{x} \in [G]^n$,
  $\lim_{y \in G} f(\vec{x}, y)$ exists. Moreover, by Wang~\cite[Lemma
  2.6]{Wang2014Some}, $C \not \leq_T G \oplus Z$. Let
  $g : [\omega]^n \to k$ be defined by
  $g(i_0, \dots, i_{n-1}) = \lim_{x_n \in G} f(x_{i_0}, \dots,
  x_{i_{n-1}}, x_n)$. By strong cone avoidance of
  $\rt^n_{<\infty, \ell}$ applied to $g$, there is an infinite set $H$
  such that $|g[H]^n| \leq \ell$ and such that
  $C \not \leq_T H \oplus G \oplus Z$.  The set $H \oplus G \oplus Z$
  computes an infinite set $S$ such that $|f[S]^{n+1}| \leq \ell$. In
  particular, $C \not \leq_T S \oplus Z$.
	
  $\Leftarrow$: Fix $C \not \leq_T Z$ and an arbitrary coloring
  $f : [\omega]^n \to k$. By Patey~\cite[Theorem
  2.6]{PateyCombinatorial}, there is a set $B$ such that $f$
  is $\Delta^0_2(B)$ and $C \not \leq_T B \oplus Z$. By Shoenfield's
  limit lemma~\cite{Shoenfield1959degrees}, there is a $B$-computable
  coloring $g : [\omega]^{n+1} \to k$ such that for every
  $\vec{x} \in [\omega]^n$, $\lim_y g(\vec{x}, y) = f(\vec{x})$. By
  cone avoidance for $\rt^{n+1}_{<\infty, \ell}$, there is an infinite set
  $H$ such that $|g[H]^{n+1}| \leq \ell$ and such that
  $C \not \leq_T B \oplus H \oplus Z$. In particular,
  $|f[H]^n| \leq \ell$. This completes the proof.
\end{proof}

Strong cone avoidance has therefore two main interests: First, it
gives some insight about the combinatorial nature of a problem $\Psf$,
by expressing the inability of the problem $\Psf$ to code some fixed
set even with an arbitrary instance. Second, it can be used as a tool
to prove that a problem does not imply $\aca$ over~$\rca$.

Wang proved that for every $n \geq 1$ and every sufficiently large
$\ell$, $\rt^n_{<\infty, \ell}$ admits strong cone avoidance, and, in
particular, cone avoidance. On the other hand, Dorais et
al.~\cite{Dorais2016uniform} proved that for every $n \geq 3$,
$\rt^n_{<\infty, 2^{n-2}-1}$ does not admit cone avoidance.

\begin{remark}
  One of the goals of this paper is to explicitly determine these
  bounds.  It turns that \emph{Catalan} and \emph{Schr\"oder} numbers
  are involved.  There are over $300$ interpretations of Catalan
  numbers, and a few dozen of Schr\"{o}der numbers.  The easiest way
  to compare them is: The $n$th Catalan number is the number of paths
  from $(0,0)$ to $(n,n)$ that take steps $(0,1)$ and $(1,0)$, and
  don't go above main diagonal; the $n$th Schr\"{o}der number is the
  same, except the paths are also allowed to take $(1,1)$ steps.  The
  best current references are \url{http://oeis.org/A000108} and
  \url{http://oeis.org/A006318}. Another reference is Stanley
  \cite{MR3467982}.
\end{remark}

The explicit bound given by Wang is the sequence of {Schr\"oder
  numbers}, which starts with
$1, 2, 6, 22, 90, 394, 1806, 8558, \dots$ and grows faster than the
lower bound of Dorais et al. In particular, this left open whether
$\rt^3_{<\infty, 5}$ and $\rt^3_{<\infty, 4}$ admit cone avoidance.

\subsection{Summary of our results}

We now give a summary on the known bounds on the threshold between
strong cone avoidance of $\rt^n_{<\infty, \ell}$ and the existence of
an instance of $\rt^n_{<\infty, \ell}$ all of whose solutions are
above a fixed cone.  We say that a set $A$ is
\emph{$\rt^n_{<\infty, \ell}$-encodable} if there is an instance of
$\rt^n_{<\infty, \ell}$ such that every solution computes~$A$.

Let $d_0, d_1, \dots$ be the sequence of Catalan numbers inductively
defined by $d_0 = 1$ and
$$
d_{n+1} = \sum_{i=0}^n d_i d_{n-i}
$$
In particular, $d_0 = 1$, $d_1 = 1$, $d_2 = 2$, $d_3 = 5$, $d_4 = 14$,
$d_5 = 42$, $d_6 = 132$, $d_7 = 429$.

\begin{theorem}The $\rt^n_{<\infty, \ell}$-encodable sets are
  precisely
  \begin{itemize}
  \item[(a)] the hyperarithmetic sets if and only if $\ell < 2^{n-1}$
  \item[(b)] the arithmetic sets if and only if
    $2^{n-1} \leq \ell < d_n$
  \item[(c)] the computable sets if and only if $d_n \leq \ell$
  \end{itemize}
\end{theorem}
\begin{proof}
  (a) By Solovay~\cite{Solovay1978Hyperarithmetically}, any
  $\rt^n_{<\infty, \ell}$-encodable set must be hyperarithmetical. By
  Theorem~\ref{modulus}, every hyperarithmetical set $A$ must have a
  modulus $\mu_A$ (also see Definition~\ref{dmodulus}).  Apply
  Theorem~\ref{thm:thin-set-dominating-arbitrary-g} to $\mu_A$ to get
  a coloring $f:[\omega]^n\rightarrow 2^{n-1}$ where every infinite
  $f$-thin computes a function which dominates $\mu_A$ and hence
  computes $A$.

  %we prove that
  % every hyperarithmetical set is $\rt^n_{<\infty, \ell}$-encodable if
  % $\ell < 2^{n-1}$.

  (b) In Theorem~\ref{thm:rtn-strong-cone-avoidance-arith}, we prove
  that any $\rt^n_{<\infty, \ell}$-encodable set must be arithmetical
  whenever $\ell \geq 2^{n-1}$. On the other hand, we prove in
  Theorem~\ref{thm:rtn-strong-cone-avoidance-arith} that every
  arithmetical set is $\rt^n_{<\infty, \ell}$-encodable if
  $\ell < d_n$.

  (c) In Theorem~\ref{thm:rtn-strong-cone-avoidance-real}, we prove
  that any $\rt^n_{<\infty, \ell}$-encodable set must be computable
  whenever $\ell \geq d_n$. Of course, every computable set is
  trivially $\rt^n_{<\infty, \ell}$-encodable.
\end{proof}

\begin{figure}[htbp]
  \begin{center}
    \def\arraystretch{1.5}
    \begin{tabular}{|c|c|c|c|c|}
      \hline
      Strongly computes & hyperarithmetic & arithmetic & computable\\ \hline
      $\rt^1_{<\infty, \ell}$ & & & $\ell \geq 1$\\
      $\rt^2_{<\infty, \ell}$ & $\ell = 1$ & & $\ell \geq 2$\\
      $\rt^3_{<\infty, \ell}$ & $\ell \leq 3$ & $\ell = 4$ & $\ell \geq 5$\\ 
      $\rt^4_{<\infty, \ell}$ & $\ell \leq 7$ & $8 \leq \ell \leq 13$ & $\ell \geq 14$\\
      $\rt^5_{<\infty, \ell}$ & $\ell \leq 15$ & $16 \leq \ell \leq 41$ & $\ell \geq 42$\\ \hline
    \end{tabular}
  \end{center}
  \caption{This table gives a summary of the class of
    $\rt^n_{<\infty, \ell}$-encodable sets in function of the size $n$
    of the tuples. For example, the $\rt^2_{<\infty, 4}$-encodable
    sets are exactly the arithmetical sets.}
\end{figure}

In terms of (strong) cone avoidance, we obtained the following
characterization.

\begin{theorem}
  For $n \geq 1$, $\rt^n_{<\infty, \ell}$ admits
  \begin{itemize}
  \item[(a)] strong cone avoidance if and only if $\ell \geq d_n$
  \item[(b)] cone avoidance if and only if $\ell \geq d_{n-1}$
  \end{itemize}
\end{theorem}
\begin{proof}
  By Theorem~\ref{thm:rtn-strong-cone-avoidance-real},
  Theorem~\ref{thm:rtn-strong-cone-avoidance-arith} and
  Theorem~\ref{thm:bridge-strong-to-non-strong}.
\end{proof}

\subsection{A brief history of the Thin Set Theorems}

Basically the idea of a thin set for a coloring is a set where the
number of colors used is less than the number of colors available to
the coloring. This idea was first raised by Friedman
~\cite{FriedmanFom53free}.  The first paper about thin sets was in
Cholak et al.~\cite{Cholak2001Free} but in a different form than the
one presented here.  As we mentioned above, Wang~\cite{Wang2014Some}
first studied the thin set theorem in the form we are considering now.
He used the name Achromatic Ramsey Theorem.  The lower bounds we
exploit appeared in Dorais et al.~\cite{Dorais2016uniform}.
Montalban~\cite{Montalban2011Open} asked about the reverse mathematics
of $\rt^n_{k,\ell}$.  %% Ludovic, know what question?
Patey \cite{Patey2016weakness} showed this hierarchy is strictly
decreasing over $\rca$, by proving that the more colors you allow in
the solution, the more non-c.e.\ definitions you can preserve
simultaneously. This is the first known strictly decreasing hierarchy.
Patey showed the same result in \cite{Patey2015Iterative} with
preservation of hyperimmunity.

There are two upcoming papers where our results play a role.  Downey
et al.~\cite{Downey:2019uf} where they prove that (strong)
preservation of $1$-hyperimmunity is the same as (strong) cone
avoidance.  Patey~\cite{Patey:2019wd} where the results here on the
thin set theorem are used as a blackbox to decide (strong) cone
avoidance for a whole class of Ramsey-like theorems. In some sense,
this shows that the thin set theorem is ``analysis-complete", in that
its analysis contains the information to understand all the other
Ramsey-like problems.

\section{Thin set theorems and sparsity}

Some degrees of unsolvability can be described by the ability to
compute fast-growing functions. For example, a Turing degree is
\emph{hyperimmune} if it contains a function which is not computably
dominated. By Martin's domination theorem, a Turing degree is
\emph{high}, that is, $\mathbf{d}' \geq \mathbf{0}''$ if it contains a
function dominating every computable function.  The notion of modulus
establishes a bridge between the ability to compute fast-growing
functions and the ability to compute sets.

\begin{definition}\label{dmodulus}
  A function $g$ \emph{dominates} a function $f$ if $g(x) \geq f(x)$
  for every $x$.  A function $\mu_X : \omega \to \omega$ is a
  \emph{modulus} for $X$ if every function dominating $\mu_X$
  computes~$X$.
\end{definition}

In the case of Ramsey's theorem and more generally Ramsey-type
theorems, one usually proves lower bounds by constructing an instance
such that every solution $H$ will be sufficiently sparse, so that its
\emph{principal function} $p_H$ is sufficiently fast-growing.

\begin{definition}
  The \emph{principal function} of an infinite set
  $X = \{x_0 < x_1 < \dots \}$ is the function $p_X$ defined by
  $p_X(n) = x_n$.
\end{definition}

Consider for example Ramsey theorem for pairs and two colors
($\rt^2_2$).  Given an arbitrary function $g : \omega \to \omega$, one
can define a function $f : [\omega]^2 \to 2$ by $f(x, y) = 1$ if
$g(x) \leq y$ and $f(x, y) = 0$ otherwise. Then every infinite
$f$-homogeneous set $H$ will be of color 1, and the principal function
$p_H$ will dominate $g$.  By taking $g$ to be a modulus for the
halting set, this proves that $\rt^2_2$ does not admit strong cone
avoidance.

\subsection{Ramsey-type theorems and sparsity}

For many theorems in Ramsey Theory, the class of solutions to each
instance always has certain structural features. 
% Many theorems coming from Ramsey theory share some structural
% features considering the class of the solutions of a given instance.
In what follows, $[X]^\omega$ denotes the class of all
infinite subsets of~$X$.

\begin{definition}
  A class $\Ccal \subseteq [\omega]^\omega$ is \emph{dense} if
  $(\forall X \in [\omega]^\omega)[X]^\omega \cap \Ccal \neq
  \emptyset$, and is \emph{downward-closed} if
  $(\forall X \in \Ccal)[X]^\omega \subseteq \Ccal$.  A class
  $\Ccal \subseteq [\omega]^\omega$ is \emph{Ramsey-like} if it is
  dense and downward-closed.
\end{definition}

One can easily check that given an instance $f : [\omega]^n \to k$ of
the thin set theorems, and some $\ell$, the collection of all infinite
sets $H$ such that $|f[H]^n| \leq \ell$ is a Ramsey-like class. Since
every Ramsey-like class is closed under subset, one can intuitively
only compute with positive information, in that the absence of an
integer in a solution $H$ is not informative. It is therefore natural
to conjecture that the only computational power of Ramsey-type
theorems comes from the sparsity of their solutions, and thus from
their ability to compute fast-growing functions.

The first result towards this intuition is the characterization of the
sets which are computed by a downward-closed class, as those which
admit a modulus function.  We say that a set $A$ is \emph{computably
  encodable} if every infinite set has an infinite subset which computes $A$ (so
there is a dense class $\Ccal \subseteq [\omega]^\omega$, all of whose
elements compute $A$).

% there is a downward-closed class $\Ccal \subseteq [\omega]^\omega$,
% all of whose elements compute $A$.
% A set $A$ is \emph{hyperarithmetic} if it is a $\Sigma^1_1$
% singleton.

\begin{theorem}\label{modulus}
  [Solovay~\cite{Solovay1978Hyperarithmetically}, Groszek and
  Slaman~\cite{Groszek2007Moduli}] Given a set $A$, the following are
  equivalent
  \begin{itemize}
  \item[(a)] $A$ is computably encodable
  \item[(b)] $A$ is hyperarithmetic
  \item[(c)] $A$ admits a modulus
  \end{itemize}
\end{theorem}

In particular, since the set of solutions of an instance $f$ of
$\rt^n_{<\infty, \ell}$ is Ramsey-like, if every solution to $f$
computes a set $A$, then $A$ is computably encodable, hence admits a
modulus.

One can obtain a more precise result when the computation is
uniform.\footnote{The authors thank Lu Liu for letting them include
  Theorem~\ref{thm:dominating-solution-ramsey-type} and
  Theorem~\ref{thm:ramsey-type-non-uniform-computation-fails} in the
  paper.} A function $g$ is a \emph{uniform modulus} for a set $A$ if
there is a Turing functional $\Phi$ such that $\Phi^f = A$ for every
function $f$ dominating $g$. The following proposition shows that the
uniform computation of a set by a Ramsey-type class can only be done
by the sparsity of its members.

\begin{theorem}[Liu and Patey]\label{thm:dominating-solution-ramsey-type}
  Fix a set $A$ and a Ramsey-like class $\Ccal$.  If there is a Turing
  functional $\Phi$ such that $\Phi^H = A$ for every $H \in \Ccal$,
  then $p_H$ is a uniform modulus of $A$ for every $H \in \Ccal$.
\end{theorem}
\begin{proof}
  Let $H \in \Ccal$ and $f$ be a function dominating $p_H$.  For every
  $x \in \omega$ and $v \in \{0,1\}$, let $\Ical_{x, v}$ be the
  $\Pi^{0,f}_1$ class of all sets $G$ such that $p_G$ is dominated by
  $f$, and such that for every $E \subseteq G$,
  $\Phi^E(x)\downarrow \rightarrow \Phi^E(x) = v$.  Note that
  $\Ical_{x,v}$ is $\Pi^{0,f}_1$ uniformly in $f$, $x$ and $v$.  It
  follows that the set $W = \{ (x, v) : \Ical_{x,v} = \emptyset \}$ is
  $f$-c.e.\ uniformly in $f$.  Since $\Ccal$ is closed downward,
  $H \in \Ical_{x,A(x)}$, hence $(x, A(x)) \not \in W$ for every $x$.
  We have two cases.
  \begin{itemize}
  \item Case 1: There is some $x$ such that
    $(x, 1-(A(x))) \not \in W$.  Let $G \in \Ical_{x, 1-A(x)}$.  Since
    $\Ccal$ is dense, there is some $K \in [G]^\omega \cap \Ccal$, and
    by definition of $\Ical_{x, 1-A(x)}$,
    $\Phi^K(x)\downarrow \rightarrow \Phi^K(x) = 1-A(x)$, but
    $\Phi^K(x) = A(x)$ by assumption. Contradiction.
  \item Case 2: For every $x$, $(x, 1-A(x)) \in W$. Since
    $(x, A(x)) \not \in W$, we can $W$-compute $A$, and this uniformly
    in $f$.
  \end{itemize}
  This completes the proof.
\end{proof}

However, the previous theorem cannot be generalized to non-uniform
computations. In particular, there exists a set $A$ and an instance
$f$ of $\rt^2_2$ such that every solution to $f$ computes $A$, but not
through sparsity.

\begin{theorem}[Liu and Patey]\label{thm:ramsey-type-non-uniform-computation-fails}
  Let $A$ be a hyperarithmetical set.  There exists a function
  $f: [\omega]^2 \to 2$ such that
  \begin{itemize}
  \item[(a)] Every infinite $f$-homogeneous set computes $A$
  \item[(b)] For each $i < 2$, there is an infinite $f$-homogeneous
    set $H$ for color $i$ such that $p_H$ is dominated by a computable
    function.
  \end{itemize}
\end{theorem}
\begin{proof}
  Let $g : \omega \to \omega$ be a modulus for $A$.  Let
  $h : 2^{<\omega} \to \omega$ be a computable bijection.  Let
  $\tilde{f} : [\omega]^2 \to 2$ be defined by
$$
\tilde{f}(x, y) = \cond{
  0 & \mbox{ if } x = h(\sigma) \mbox{ with } \sigma \prec A \mbox{ or } y < g(x)\\
  1 & \mbox{ otherwise } }
$$
We claim that every infinite $\tilde{f}$-homogeneous set computes $A$.
If $H$ is an infinite $\tilde{f}$-homogeneous set for color 0, then
$H \subseteq \{ h(\sigma) : \sigma \prec A \}$, in which case $H$
computes $A$.  If $H$ is an infinite $\tilde{f}$-homogeneous set for
color 1, then $p_H \geq g$, and again, $H$ computes $A$.

$$
f(x, y) = \cond{
  \tilde{f}(x_1, y_1) & \mbox{ if } x = 2x_1 \mbox{ and } y_1 = \lfloor y/2\rfloor \\
  1- \tilde{f}(x_1, y_1) & \mbox{ if } x = 2x_1+1 \mbox{ and } y_1 = \lfloor y/2\rfloor\\
}
$$
We claim that every infinite $f$-homogeneous set computes an infinite
$\tilde{f}$-homogeneous set. Let $H$ be an infinite $f$-homogeneous
set for some color $i < 2$.  One of the two sets is infinite:
$$
H_0 = \{x_1 : 2x_1 \in H\} \mbox{ and } H_1 = \{x_1 : 2x_1+1 \in H\}
$$
Moreover, for every $x_1 < y_1 \in H_0$,
$\tilde{f}(x_1, y_1) = f(2x_1, 2y_1) = i$ and for every
$x_1, y_1 \in H_1$, $\tilde{f}(x_1, y_1) = 1-f(2x_1+1, 2y_1+1) = 1-i$.
Therefore, both $H_0$ and $H_1$ are $\tilde{f}$-homogeneous.

Last, we prove (b). Let
$$
G_0 = \{2h(\sigma) : \sigma \prec A\} \mbox{ and } G_1 =
\{2h(\sigma)+1 : \sigma \prec A\}
$$
For each $i < 2$, $G_i$ is an infinite $f$-homogeneous set for color
$i$.  Moreover, $p_{G_0}$ and $p_{G_1}$ are both dominated by the
computable function which to $n$, associates
$\max\{ 2h(\sigma)+1 : |\sigma| = n\}$.  This completes the proof.
\end{proof}

Last, if we restrict ourselves to computable instances, one can prove
that even non-uniform computation is done by sparsity.  The following
lemma tells us in some sense that whenever considering computable
instances of $\rt^n_{<\infty,\ell}$, the analysis of the functions the
thin sets dominate can be done, without loss of generality, by the
study of the sparsity of these thin sets.

\begin{definition}
  Given a function $g : \omega \to \omega$, a \emph{domination
    modulus} is a function $\nu_g : \omega \to \omega$ such that every
  function dominating $\nu_g$ computes a function dominating~$g$.
\end{definition}

\begin{theorem}
  Fix a function $g : \omega \to \omega$.  Let $f : [\omega]^n \to k$
  be a computable instance of $\rt^n_{<\infty,\ell}$ such that every
  solution computes a function dominating $g$. Then for every solution
  $H$, $p_H$ is a domination modulus for $g$.
\end{theorem}
\begin{proof}
  Fix $g$ and $f$. Suppose for the sake of contradiction that there is
  an infinite set $H$ such that $|f[H]^n|\leq \ell$ and such that
  $p_H$ is not a domination modulus for~$g$. By definition, there is a
  function~$h$ dominating $p_H$ such that $h$ does not computes a
  function dominating~$g$.  Let $T \subseteq \omega^{<\omega}$ be the
  $h$-computably bounded tree defined by
	$$
	T = \left\{ \sigma \in \omega^{<\omega} : (\forall x <
          |\sigma|)\sigma(x) < h(x)) \wedge |f[\sigma]^n|\leq \ell
        \right\}
	$$

        In particular, $H \in [T]$, so the tree is infinite. Moreover,
        any infinite path through $T$ is an
        $\rt^n_{<\infty,\ell}$-solution to $f$. By the
        hyperimmune-free basis theorem~\cite{Jockusch197201} relative
        to $h$, there is an infinite set $S \in [T]$ such that every
        $S$-computable function is dominated by an $h$-computable
        function. In particular, $S$ does not compute a function
        dominating~$g$. Contradiction.
      \end{proof}

      % \section{Background and definitions}\label{S:defns}
      % \input{parts/part1-definitions}

      % \section{Thin sets and cone avoidance}\label{S:cone}
      % \input{parts/part2-cone-avoidance}

      \section{The strength of the thin set
        theorems}\label{sect:strength-ts}

      In this section, we study the ability of thin set theorems to
      compute fast-growing functions. More precisely, given a fixed
      function $g : \omega \to \omega$ and some $n \geq 1$, we
      determine the largest number of colors $\ell$ such that there
      exists an instance of $\rt^n_{<\infty,\ell}$ such that every
      solution computes a function dominating $g$. Thanks to the
      notion of modulus, we will then apply this analysis to determine
      which sets can be computed by $\rt^n_{<\infty,\ell}$ whenever
      $n, \ell \geq 1$.

      As explained, one approach to prove that $\rt^n_{<\infty,\ell}$
      implies the existence of fast-growing functions is to define a
      coloring $f : [\omega]^n \to k$ such that every solution $H$ is
      sparse, and then use the principal function of~$H$.

\begin{definition}
  Given a function $g : \omega \to \omega$, an interval $[a, b]$ is
  \emph{$g$-large} if $b \geq g(a)$.  Otherwise, it is
  \emph{$g$-small}. By extension, we say that a finite set $F$ is
  $g$-large ($g$-small) if $[\min F, \max F]$ is $g$-large
  ($g$-small).
\end{definition}

Given a function $f : [\omega]^n \to k$, we say that a set
$H \subseteq \omega$ is \emph{$f$-thin} if $|f[H]^n|\leq k-1$.

\begin{theorem}\label{thm:thin-set-dominating-arbitrary-g}
  For every function $g : \omega \to \omega$ and every $n \geq 1$,
  there is a $g$-computable function $f : [\omega]^n \to 2^{n-1}$ such
  that every infinite $f$-thin set computes a function dominating~$g$.
\end{theorem}
\begin{proof}
  We prove this by induction over $n \geq 1$. Then case $n = 1$
  vacuously holds, since $f_1 : \omega \to \{\langle\rangle\}$ has no
  infinite $f_1$-thin set. Furthermore, assume without loss of
  generality that $g$ is increasing.  Given
  $x_0 < x_1 < \dots < x_{n-1}$, let
$$
f_n(x_0, \dots, x_{n-1}) = \langle gap(x_0, x_1), gap(x_1, x_2),
\dots, gap(x_{n-2}, x_{n-1})\rangle
$$
where $gap(a,b) = \ell$ if $[a,b]$ is $g$-large, and $gap(a,b) = s$
otherwise.  Let $H$ be an infinite $f_n$-thin set, say for
color~$\langle j_0, \dots, j_{n-2} \rangle$. We have several cases.

Case 1: $H$ is $f_n$-thin for color $\langle s, s, \dots, s
\rangle$. Then for every $x_0 < \dots < x_{n-1} \in H$,
$[x_0, x_{n-1}]$ is $g$-large since $g$ is increasing. Then the
function which to $i$ associates the $(i+n)$th element of $H$
dominates~$g$.

Case 2: $H$ is $f_n$-thin for color
$\langle j_0, \dots, j_{i-1}, \ell, s, s, s, \dots, s \rangle$.  Case
2.1: $H$ is $f_{i+1}$-thin for color
$\langle j_0, \dots, j_{i-1} \rangle$ (where
$\langle j_0, \dots, j_{i-1} \rangle = \langle \rangle$ if $i =
0$). Then by induction hypothesis, $H$ computes a function
dominating~$g$.  Case 2.2: There is some
$x_0 < x_1 < \dots < x_i \in H$ such that
$f_{i+1}(x_0, \dots, x_i) = \langle j_0, \dots, j_{i-1} \rangle$. Let
$t > x_i$ be such that $gap(x_i, t) = \ell$. We claim that for every
tuple $x_{i+1} < \dots < x_{n-1} \in H - \{0, \dots, t\}$,
$[x_{i+1}, x_{n-1}]$ is $g$-large. Indeed, since $x_{i+1} > t$, then
$gap(x_i, x_{i+1}) = \ell$, and by choice of $i$, $H$ is $f_n$-thin
for color~$\langle j_0, \dots, j_{i-1}, \ell, s, \dots, s \rangle$. So
all the intervals cannot be small after $x_{i+1}$, and since $g$ is
increasing, $[x_{i+1}, x_{n-1}]$ is $g$-large. The function which to
$u$ associates the $(u+n-i)$th element of $H - \{0, \dots, t\}$
dominates~$g$.  This completes the proof of
Theorem~\ref{thm:thin-set-dominating-arbitrary-g}.
\end{proof}

\begin{corollary}[Dorais et al.~\cite{Dorais2016uniform}]
  For every $n \geq 2$ and $k \geq 1$, there is a computable instance
  of $\rt^{n+k}_{<\infty, 2^{n-1}}$ such that every solution
  computes $\emptyset^{(k)}$.
\end{corollary}
\begin{proof}
  Let $g : \omega \to \omega$ be a $\emptyset^{(k)}$-computable
  modulus of $\emptyset^{(k)}$.  By
  Theorem~\ref{thm:thin-set-dominating-arbitrary-g}, there is a
  $\emptyset^{(k)}$-computable function $f = [\omega]^n \to 2^{n-1}$
  such that every infinite $f$-thin set computes a function dominating
  $g$, hence computes $\emptyset^{(k)}$.  By Schoenfield's limit
  lemma, there is a computable function
  $h : [\omega]^{n+k} \to 2^{n-1}$ such that for every
  $x_0 < \dots < x_{n-1} \in \omega$,
$$
f(x_0, \dots, x_{n-1}) = \lim_{x_n} \dots \lim_{x_{n+k-1}} h(x_0,
\dots, x_{n+k-1})
$$
Every infinite $h$-thin set is $f$-thin, and therefore computes
$\emptyset^{(k)}$.
\end{proof}

One can wonder about the optimality of
Theorem~\ref{thm:thin-set-dominating-arbitrary-g}.  In particular, for
$n = 3$, given a function $g : \omega \to \omega$, is there a function
$f : [\omega]^3 \to 5$ such that every infinite $f$-thin set computes
a function dominating~$g$?

When considering the function constructed in the proof of
Theorem~\ref{thm:thin-set-dominating-arbitrary-g}, there seems at
first sight to be some degree of freedom. In particular, given some
$a < b < c$, whenever $[a, b]$ and $[b, c]$ are both $g$-small,
$[a,c]$ can be either $g$-large or $g$-small. A candidate function
would therefore be
$$
f(a, b, c) = \langle gap(a,b), gap(b,c), gap(a,c) \rangle
$$
However, as we shall see in Section~\ref{sect:gap-principle}, the last
component $gap(a,c)$ is of no help to compute fast-growing functions,
in the sense if $g$ is not dominated by any computable function, then
one can avoid the color $\langle s, s, \ell \rangle$ without
dominating~$g$.

\begin{definition}
  A function $f$ is \emph{$X$-hyperimmune} if it is not dominated by
  any $X$-computable function.
\end{definition}

\begin{lemma}
  If $g$ is a hyperimmune function, then there is an infinite set $H$
  such that $\langle s, s, \ell \rangle \not \in f[H]^3$ and such that
  $g$ is $H$-hyperimmune.
\end{lemma}
\begin{proof}
  See Theorem~\ref{thm:gap-preserve-hyperimmune-and-non-dnc}.
\end{proof}

\begin{figure}[htbp]
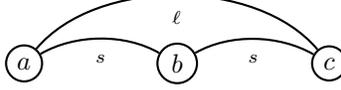

$$
\psmatrix[colsep=1.5cm,rowsep=1.5cm,mnode=circle] a&b&c
\everypsbox{\scriptstyle} \ncarc[arcangle=30]{1,1}{1,2}_{s}
\ncarc[arcangle=30]{1,2}{1,3}_{s}
\ncarc[arcangle=50]{1,1}{1,3}_{\ell} \endpsmatrix
$$
\caption{The following pattern can be avoided thanks to the $\gap$
principle studied in Section~\ref{sect:gap-principle}.}
\end{figure}

Actually, we shall see that in Section~\ref{sect:weakness-ts} that
Theorem~\ref{thm:thin-set-dominating-arbitrary-g} is optimal in the
sense that for every function $g$ which is hyperimmune relative to
every arithmetical set, for every $n \geq 1$ and every instance of
$\rt^n_{2^{n-1}+1, 2^{n-1}}$, there is a solution which does not
compute a function dominating~$g$. We can however obtain better
results in the case of left-c.e.\ functions.

\begin{definition}
  A function $g : \omega \to \omega$ is \emph{left-c.e.}\ if there is
  an uniformly computable sequence of functions $g_0, g_1, \dots$ such
  that for every $x$, $g_0(x), g_1(x), \dots$ is a non-decreasing
  sequence converging to $g(x)$.
\end{definition}

We want our approximations $g_0, g_1, \dots$ to $g$ to have some nice
properties, as enumerated by the next lemma.

\begin{lemma}\label{ceapproximations}
  Let $g_0, g_1 \ldots$ be an uniformly computable sequence of
  functions such that $\lim_s g_s(x) = g(x)$.  Then there is a
  uniformly computable sequence of functions of
  $\hat{g}_0, \hat{g}_1, \ldots$ such that
  \begin{enumerate}
  \item For all $x$, $\lim_s \hat{g}_s(x)$ exists.
  \item $\lim_s \hat{g}_s(x) \geq g(x)$. So $\lim_s \hat{g}_s(x)$
    dominates $g$.
  \item $\lim_s \hat{g}_s(x)$ is increasing.
  \item Each $\hat{g}_s$ is increasing.
  \item For all $x$, the sequence
    $\hat{g}_0(x), \hat{g}_1(x), \hat{g}_2(x), \ldots$ is
    non-decreasing.
  % \item For all $a \leq x$, if $\hat{g}_s(a) < \hat{g}_{s+1}(a)$, then
  %   $\hat{g}_s(x+1) < \hat{g}_{s+1}(x+1)$.
    \item If $a\leq x$, $\hat{g}_s(a) <x+1$, and
    $\hat{g}_{s+1}(a) \geq x+1$ then $\hat{g}_{s+1}(x+1) \geq s+1 \geq x+1$.
  \end{enumerate}
\end{lemma}

\begin{proof}
  Let $\hat{g}_0(0)= g_0(0)$; $\hat{g}_0(x+1)$ be the max of
  $g_0(x+1)$ and $\hat{g}_0(x)+1$ (so $\hat{g}_0$ is increasing); and
  $\hat{g}_{s+1}(0)$ be the max of $\hat{g}_{s}(0)$ and $g_{s+1}(0)$
  (so the sequence $\hat{g}_0(0), \hat{g}_1(0), \hat{g}_2(0), \ldots$
  is non-decreasing and $\lim_s \hat{g}_s(0)\geq g(0)$)).  Let
  $\hat{g}_{s+1}(x+1)$ be the max of $\hat{g}_{s+1}(x)+1$ (so
  $\hat{g}_{s+1}$ is increasing), $\hat{g}_{s}(x+1)$ (so the sequence
  $\hat{g}_{0}(x), \hat{g}_1(x), \hat{g}_2(x), \ldots$ is
  non-decreasing), $g_{s+1}(x+1)$ (so
  $\lim_s \hat{g}_s(x)\geq g(x+1)$),
  % $\hat{g}_{s}(x+1)+1$ but only if
  % there is $a \leq x$ such that $\hat{g}_s(a) < \hat{g}_{s+1}(a)$
  % (so
  % second to last condition is meet),
  and also the $s$'s but only if there exists $a \leq x$ such that
  $\hat{g}_s(a) < x+1$ and $\hat{g}_{s+1}(a) \geq x+1$ (so last
  condition is meet).  By induction on $x$, $\lim_s \hat{g}_s(x)$
  exists.  Since each $g_s$ is increasing, so is $\lim_s \hat{g}_s(x)$.
\end{proof}

Since we are only interested in functions which dominate $g$ (see
Theorem~\ref{thm:rt354-domination-left-ce}), the above lemma says that
we can make the assumption all of our approximations have the
properties enumerated in the above lemma and that $g$ is increasing.
Also note that the resulting function $\lim_s \hat{g}_s(x)$ is
left-c.e.\ even in the case of an arbitrary $\Delta^0_2$ function
$g$. Here however we shall only consider only left-c.e.\ functions
$g$.

% \begin{definition}
%   Let $g : \omega \to \omega$ be a left-c.e.\ increasing function with
%   computable approximations $g_0, g_1, \dots$ An interval $[a, b]$ is
%   \emph{$g_c$-large} if $b \geq g_c(a)$.  Otherwise, it is
%   \emph{$g_c$-small}.
% \end{definition}

Since $g_s$ is computable, being $g_s$-large or $g_s$-small is
decidable, contrary to $g$-largeness and $g$-smallness.  If $[a, b]$
is $g$-large, then, for all $s$, since $g_s(a) \leq g(a)$, $[a,b]$ is
$g_s$-large.  The interesting feature of left-c.e.\ functions is that
the set of their small intervals is c.e.  Indeed, $[a, b]$ is
$g$-small iff, there is some $s \in \omega$ such that $g_s(a) > b$,
i.e.\ $[a,b]$ is $g_s$-small.  Therefore, this notion is interesting
to classify $g$-small sets.  The last condition of the above lemma and
the fact that $g_s$ is increasing imply that if $s$ is the first stage
where $[a,b]$ is $g_s$-small then, for all $y\geq b$,
$g_{g_s(y)}(a)>b$.  So $g$ grows sufficiently fast so that if $[a, b]$
is $g$-small and $[b, c]$ is $g$-large, then $[a, b]$ is $g_c$-small.

Before proving our main lower bound theorem, we need to introduce a
combinatorial tool, namely, \emph{largeness graphs}. They will be
useful to count the number of colors needed for our theorem.

Given a tuple $x_0 < x_1 < \dots < x_{n-1}$, we study the properties
of the graph whose vertices are $\{0, \dots, n-1\}$, and such that for
every $i < n-1$, there is an edge between $i$ and $i+1$ if
$[x_i, x_{i+1}]$ is $g$-large, and for every $i +1 < j < n$, there is
an edge between $i$ and $j$ if $[x_i, x_{i+1}]$ is
$g_{x_j}$-small. This yields the notion of a largeness graph.

\begin{definition}
  A \emph{largeness graph} of size $n$ is an undirected irreflexive
  graph $(V, E)$, with $V = \{0, \dots, n-1\}$, such that
  \begin{itemize}
  \item[(a)] If $\{i,i+1\} \in E$, then for every $j > i+1$,
    $\{i, j\} \not \in E$
  \item[(b)] If $i < j < n$, $\{i,i+1\} \not \in E$ and
    $\{j,j+1\} \in E$, then $\{i, j+1\} \in E$
  \item[(c)] If $i +1 < j < n - 1$ and $\{i,j\} \in E$, then
    $\{i, j+1\} \in E$
  \item[(d)] If $i+1 < j < k < n$ and $\{i, j\} \not \in E$ but
    $\{i,k\} \in E$, then $\{j-1, k\} \in E$
  \end{itemize}
\end{definition}

Property (a) reflects the fact that if $[x_i, x_{i+1}]$ is $g$-large,
then it is not $g_{x_j}$-small for any $j > i+1$. Property (b) says
that if $[x_j, x_{j+1}]$ is $g$-large, then any value larger than
$x_{j+1}$ will already witness the smallness of all the $g$-small
intervals before $x_j$.  Property (c) says that if $[x_i, x_{i+1}]$ is
$g_{x_j}$-small, then it will be $g_{x_k}$-small for every $k \geq
j$. Last, Property (d) says that if $[x_i,x_{i+1}]$ is
$g_{x_k}$-small, but $g_{x_j}$-large, then the interval
$[x_{j-1}, x_j]$ is $g_{x_k}$-small. Actually, by (a), we already know
that $[x_{j-1}, x_j]$ is $g$-small. What Property (d) adds is that
this smallness is witnessed by time $x_k$.  Since, for the first $t$
such that $g_t(x_i) \geq x_{i+1}$, $x_j < t \leq x_k$, the conditions
of Lemma~\ref{ceapproximations} implies that $g_t(x_{i+1}) \geq
t > x_{j}$. Since $t\leq x_k$ and the increasing properties of the
approximations, $g_{x_k}(x_{j-1}) \geq t > x_j$.

% We can ensure this extra property by changing the enumeration
% $g_0, g_1, \dots$ such that whenever $g_s(x) < g_{s+1}(x)$, then
% $g_s(y) < g_{s+1}(y)$ for every $y \geq x$.

\begin{figure}[htbp]
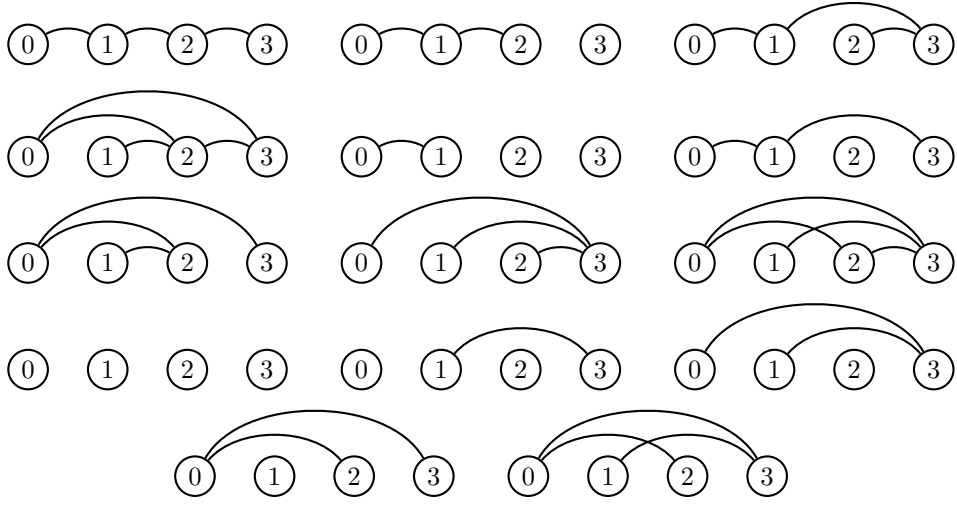

$$
\psmatrix[colsep=0.5cm,rowsep=1cm,mnode=circle] 0&1&2&3
\everypsbox{\scriptstyle} \ncarc[arcangle=30]{1,1}{1,2}
\ncarc[arcangle=30]{1,2}{1,3} \ncarc[arcangle=30]{1,3}{1,4}
\endpsmatrix \hspace{0.7cm}
\psmatrix[colsep=0.5cm,rowsep=1cm,mnode=circle] 0&1&2&3
\everypsbox{\scriptstyle} \ncarc[arcangle=30]{1,1}{1,2}
\ncarc[arcangle=30]{1,2}{1,3} \endpsmatrix \hspace{0.7cm}
\psmatrix[colsep=0.5cm,rowsep=1cm,mnode=circle] 0&1&2&3
\everypsbox{\scriptstyle} \ncarc[arcangle=30]{1,1}{1,2}
\ncarc[arcangle=30]{1,3}{1,4} \ncarc[arcangle=50]{1,2}{1,4}
\endpsmatrix
$$
\vspace{0.5cm}
$$
\psmatrix[colsep=0.5cm,rowsep=1cm,mnode=circle] 0&1&2&3
\everypsbox{\scriptstyle} \ncarc[arcangle=30]{1,2}{1,3}
\ncarc[arcangle=30]{1,3}{1,4} \ncarc[arcangle=50]{1,1}{1,3}
\ncarc[arcangle=60]{1,1}{1,4} \endpsmatrix \hspace{0.7cm}
\psmatrix[colsep=0.5cm,rowsep=1cm,mnode=circle] 0&1&2&3
\everypsbox{\scriptstyle} \ncarc[arcangle=30]{1,1}{1,2} \endpsmatrix
\hspace{0.7cm} \psmatrix[colsep=0.5cm,rowsep=1cm,mnode=circle]
0&1&2&3 \everypsbox{\scriptstyle} \ncarc[arcangle=30]{1,1}{1,2}
\ncarc[arcangle=50]{1,2}{1,4} \endpsmatrix
$$
\vspace{0.5cm}
$$
\psmatrix[colsep=0.5cm,rowsep=1cm,mnode=circle] 0&1&2&3
\everypsbox{\scriptstyle} \ncarc[arcangle=30]{1,2}{1,3}
\ncarc[arcangle=50]{1,1}{1,3} \ncarc[arcangle=60]{1,1}{1,4}
\endpsmatrix \hspace{0.7cm}
\psmatrix[colsep=0.5cm,rowsep=1cm,mnode=circle] 0&1&2&3
\everypsbox{\scriptstyle} \ncarc[arcangle=30]{1,3}{1,4}
\ncarc[arcangle=50]{1,2}{1,4} \ncarc[arcangle=60]{1,1}{1,4}
\endpsmatrix \hspace{0.7cm}
\psmatrix[colsep=0.5cm,rowsep=1cm,mnode=circle] 0&1&2&3
\everypsbox{\scriptstyle} \ncarc[arcangle=30]{1,3}{1,4}
\ncarc[arcangle=50]{1,1}{1,3} \ncarc[arcangle=50]{1,2}{1,4}
\ncarc[arcangle=60]{1,1}{1,4} \endpsmatrix
$$
\vspace{0.5cm}
$$
\psmatrix[colsep=0.5cm,rowsep=1cm,mnode=circle] 0&1&2&3
\everypsbox{\scriptstyle} \endpsmatrix \hspace{0.7cm}
\psmatrix[colsep=0.5cm,rowsep=1cm,mnode=circle] 0&1&2&3
\everypsbox{\scriptstyle} \ncarc[arcangle=50]{1,2}{1,4} \endpsmatrix
\hspace{0.7cm} \psmatrix[colsep=0.5cm,rowsep=1cm,mnode=circle]
0&1&2&3 \everypsbox{\scriptstyle} \ncarc[arcangle=50]{1,2}{1,4}
\ncarc[arcangle=60]{1,1}{1,4} \endpsmatrix
$$
\vspace{0.5cm}
$$
\psmatrix[colsep=0.5cm,rowsep=1cm,mnode=circle] 0&1&2&3
\everypsbox{\scriptstyle} \ncarc[arcangle=50]{1,1}{1,3}
\ncarc[arcangle=60]{1,1}{1,4} \endpsmatrix \hspace{0.7cm}
\psmatrix[colsep=0.5cm,rowsep=1cm,mnode=circle] 0&1&2&3
\everypsbox{\scriptstyle} \ncarc[arcangle=50]{1,1}{1,3}
\ncarc[arcangle=50]{1,2}{1,4} \ncarc[arcangle=60]{1,1}{1,4}
\endpsmatrix
$$
\caption{List of the 14 largeness graphs of size 4. The last 5
graphs are the packed largeness graphs of size 4.}
\end{figure}

\begin{definition}
  A largeness graph $\Gcal = (\{0, \dots, n-1\}, E)$ is \emph{packed}
  if for every $i < n-2$, $\{i,i+1\} \not \in E$.
\end{definition}

\begin{definition}
  Let $\Gcal_0 = (\{0, \dots, n-1\}, E_0)$ and
  $\Gcal_1 = (\{0, \dots, n-1\}, E_1)$ be two largeness graphs of size
  $n$. We define the equivalence relation $\Gcal_0 \sim \Gcal_1$ to
  hold if for every $i +1 < j < n$, $\{i,j\} \in E_0$ if and only if
  $\{i,j\} \in E_1$.
\end{definition}

In other words, $\Gcal_0 \sim \Gcal_1$ if and only if only the edges
of the form $\{i,i+1\}$ can vary.

\begin{lemma}\label{lem:equivalence-largeness-to-packed}
  Every largeness graph of size $n \geq 1$ is equivalent to a packed
  largeness graph.
\end{lemma}
\begin{proof}
  Let $\Gcal_0 = (\{0, \dots, n-1\}, E_0)$ be a largeness graph of
  size $n$.  Let $E_1 = E_0 \setminus \{ \{i,i+1\} : i < n-1 \}$ and
  $\Gcal_1 = (\{0, \dots, n-1\}, E_1)$. We claim that $\Gcal_1$ is a
  largeness graph by checking properties (a-d). Properties (a) and (b)
  are vacuously true.  Properties (c) and (d) are inherited from
  $\Gcal_0$.  By construction, $\Gcal_0 \sim \Gcal_1$ and $\Gcal_1$ is
  packed.
\end{proof}

\begin{lemma}\label{lem:normal-largeness-graph-progress}
  Let $\Gcal_0 = (\{0, \dots, n-1\}, E_0)$ be a largeness graph of
  size $n$.  Let $\ell < n-2$ be minimal (if it exists) such that
  $\{\ell,\ell+1\} \not \in E_0$ and $\{\ell,n-1\} \not \in E_0$. Then
  the graph $\Gcal_1 = (\{0, \dots, n-1\}, E_0 \cup \{\ell,\ell+1\})$
  is a largeness graph such that $\Gcal_0 \sim \Gcal_1$.
\end{lemma}
\begin{proof}
  We check that Property (a-d) are satisfied for $\Gcal_1$.

  (a) We need to check that for every $j > \ell+1$,
  $\{\ell, j\} \not \in E_0 \cup \{\ell,\ell+1\})$. If
  $\{\ell, j\} \in E_0$ for some $j > \ell+1$, then by property (c) of
  $\Gcal_0$, $\{\ell, n-1\} \in E_0$, contradicting our hypothesis.

  (b) We need to check that if $i < \ell$ and
  $\{i, i+1\} \not \in E_0 \cup \{\ell,\ell+1\})$, then
  $\{i, \ell+1\} \in E_0 \cup \{\ell,\ell+1\})$.  By minimality of
  $\ell$, $\{i, n-1\} \in E_0$.  If
  $\{i, \ell+1\} \not \in E_0 \cup \{\ell,\ell+1\})$, then by property
  (d) of $\Gcal_0$, $\{\ell, n-1\} \in E_0$, contradicting our
  hypothesis.

  (c-d) are inherited from properties (c-d) of $\Gcal_0$.
\end{proof}

\begin{definition}
  A largeness graph $\Gcal = (\{0, \dots, n-1\}, E)$ of size
  $n \geq 2$ is \emph{normal} if $\{n-2,n-1\} \in E$.
\end{definition}

\begin{lemma}\label{lem:equivalence-normal-largeness-graph}
  Every largeness graph of size $n \geq 2$ is equivalent to a normal
  largeness graph.
\end{lemma}
\begin{proof}
  Fix a largeness graph $\Gcal_0 = (\{0, \dots, n-1\}, E_0)$.  By
  iterating Lemma~\ref{lem:normal-largeness-graph-progress}, there is
  a graph $\Gcal_1 = (\{0, \dots, n-1\}, E_1)$ equivalent to $\Gcal_0$
  such that for every $\ell < n-2$ such that
  $\{\ell,\ell+1\} \not \in E_1$, then $\{\ell,n-1\} \in E_1$. The
  graph $\Gcal_2 = (\{0, \dots, n-1\}, E_1 \cup \{n-2,n-1\})$ is a
  normal largeness graph equivalent to $\Gcal_0$.
\end{proof}

\begin{figure}[htbp]
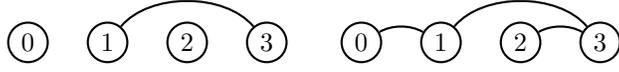

$$
\psmatrix[colsep=0.5cm,rowsep=1cm,mnode=circle] 0&1&2&3
\everypsbox{\scriptstyle} \ncarc[arcangle=50]{1,2}{1,4} \endpsmatrix
\hspace{0.7cm} \psmatrix[colsep=0.5cm,rowsep=1cm,mnode=circle]
0&1&2&3 \everypsbox{\scriptstyle} \ncarc[arcangle=30]{1,1}{1,2}
\ncarc[arcangle=30]{1,3}{1,4} \ncarc[arcangle=50]{1,2}{1,4}
\endpsmatrix
$$
\caption{Example of a non-normal largeness graph equivalent to a
normal one.}
\end{figure}

The following lemma will be very useful for counting purposes.

\begin{lemma}
  The following are in one-to-one correspondance for every $n \geq 2$:
  \begin{itemize}
  \item[(i)] packed largeness graphs of size $n$
  \item[(ii)] normal largeness graphs of size $n$
  \item[(iii)] largeness graphs of size $n-1$
  \end{itemize}
\end{lemma}
\begin{proof}
  $(i) \Leftrightarrow (ii)$: By
  Lemma~\ref{lem:equivalence-largeness-to-packed} and
  Lemma~\ref{lem:equivalence-normal-largeness-graph}, every
  equivalence class from $\sim$ contains a packed largeness graph of
  size $n$ and a normal largeness graph of size $n$. Moreover, there
  cannot be two packed largeness graphs in the same equivalence
  class. We now prove that two normal largeness graphs cannot belong
  to the same equivalence class. Indeed, let
  $\Gcal_0 = (\{0, \dots, n-1\}, E_0)$ and
  $\Gcal_1 = (\{0, \dots, n-1\}, E_1)$ be two normal largeness graphs
  of size $n$ such that $\Gcal_0 \sim \Gcal_1$ but
  $\Gcal_0 \neq \Gcal_1$. Then without loss of generality, we can
  assume there is some $i < n-1$ such that $\{i,i+1\} \in E_0$ and
  $\{i,i+1\} \not \in E_1$. By definition of normality,
  $\{n-2,n-1\} \in E_0 \cap E_1$, so by property (b) of the definition
  of a largeness graph, $\{i,n-1\} \in E_1$, and by property (a) of
  the definition of a largeness graph, $\{i,n-1\} \not \in E_0$,
  contradicting $\Gcal_0 \sim \Gcal_1$.

  $(ii) \Leftrightarrow (iii)$: Every largeness graph
  $\Gcal_0 = (\{0, \dots, n-2\}, E_0)$ of size $n-1$ can be augmented
  into a normal largeness graph $\Gcal_1 = (\{0, \dots, n-1\}, E_1)$
  of size $n$ by adding a node $n-1$ with $\{n-2,n-1\} \in E_1$. The
  edges with endpoint $n-1$ are all uniquely determined by the
  definition of a largeness graph, so the graph $\Gcal_1$ is
  unique. Conversely, given a normal largeness graph $\Gcal_1$, the
  subgraph induced by removing the last node is a largeness graph.
\end{proof}

We will now count the number of packed and general largeness graphs of
size $n$.  %For this, let us define \emph{Catalan numbers}.
Let
$d_0, d_1, \dots$ be the sequence of Catalan numbers inductively defined by
$d_0 = 1$ and
$$
d_{n+1} = \sum_{i=0}^n d_i d_{n-i}
$$
% In particular, $d_0 = 1$, $d_1 = 1$, $d_2 = 2$, $d_3 = 5$, $d_4 = 14$,
% $d_5 = 42$, $d_6 = 132$, $d_7 = 429$, $\dots$ Note that this sequence
% corresponds to the OEIS sequence A000108.

\begin{lemma}\label{lem:counting-largeness-graphs}
  For every $n \geq 0$, there exists exactly $d_n$ many largeness
  graphs of size $n$.
\end{lemma}
\begin{proof}
  By induction over $n$. By convention, there exists a unique
  largeness graph with no nodes, and we let $d_0 = 1$. Assume the
  property holds for every $j \leq n$.  Consider an arbitrary
  largeness graph of size $n+1$.  Let $i < n$ be the least index such
  that $\{i, i+1\}$ has an edge, if it exists. If there is no such
  index, then set $i = n$. We have two cases.

  Case 1: $i = 0$. Then all the edges with the endpoint $0$ are
  already specified, and the graph with nodes $\{1,\dots, n\}$ is
  unspecified. Therefore, by induction hypothesis, there are $d_n$
  many possibilities. Since $d_0 = 1$, there are $d_i \cdot d_{n-1}$
  many possibilities.

  Case 2: $0 < i \leq n$. Then by minimality of $i$, there is no edge
  $\{j,j+1\}$ for any $j < i$, and an edge $\{i,i+1\}$. By definition
  of a largeness coloring, all the edges between nodes on the left of
  $i$ and on the right of $i+1$ are fully specified. However, the
  subgraph on $\{0, \dots, i\}$ can be an arbitrary packed largeness
  graph of size $i+1$, and the subgraph on $\{i+1, \dots, n\}$ can be
  an arbitrary largeness graph of size $n-i$. By the one-to-one
  correspondence between packed largeness graphs of size $i+1$ and
  largeness graphs of size $i$, and by induction hypothesis, there are
  $d_i \cdot d_{n-i}$ many possibilities. See
  Figure~\ref{fig:example-counting-largeness-coloring}.

  Summing up the possibilities, we get
  $d_{n+1} = \sum_{i=0}^n d_i d_{n-i}$.  This completes the proof.
\end{proof}

It follows that $d_n$ is also the number of packed largeness graphs of
size $n+1$, and the number of normal largeness graphs of size $n+1$.

\begin{figure}[htbp]
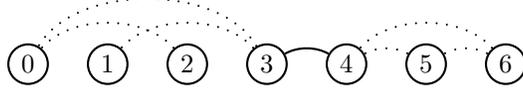

  \vspace{0.5cm}
$$
\psmatrix[colsep=0.5cm,rowsep=1cm,mnode=circle] 0&1&2&3&4&5&6
\everypsbox{\scriptstyle} \ncarc[arcangle=30]{1,4}{1,5}_{}
\ncarc[linestyle=dotted,arcangle=30]{1,5}{1,6}_{}
\ncarc[linestyle=dotted,arcangle=30]{1,6}{1,7}_{}
\ncarc[linestyle=dotted,arcangle=50]{1,1}{1,3}_{}
\ncarc[linestyle=dotted,arcangle=50]{1,2}{1,4}_{}
\ncarc[linestyle=dotted,arcangle=60]{1,1}{1,4}_{}
\ncarc[linestyle=dotted,arcangle=50]{1,5}{1,7}_{} \endpsmatrix
$$
\caption{Assuming that $\{3, 4\}$ is the left-most adjacent pair
with an edge. Then the packed largeness coloring $\{0,1,2,3\}$ and
the largeness coloring $\{4,5,6\}$ remain to be
determined. Therefore there are $d_3 \cdot d_3$ many
possibilities.}\label{fig:example-counting-largeness-coloring}
\end{figure}

We are ready to prove our main lower bound theorem.

\begin{theorem}\label{thm:rt354-domination-left-ce}
  Let $g : \omega \to \omega$ be a left-c.e.\ increasing function.
  For every $n \geq 1$, there is a $\Delta^0_2$ coloring
  $f : [\omega]^n \to d_n$ such that every infinite $f$-thin set
  computes a function dominating $g$.
\end{theorem}
\begin{proof}
  Let $\Lcal_n$ be the collection of all largeness graphs of size $n$.
  For every $n \geq 1$, we construct a function
  $f_n : [\omega]^n \to \Lcal_n$ such that every infinite $f_n$-thin
  set computes a function dominating $g$.  By
  Lemma~\ref{lem:counting-largeness-graphs}, the range of $f_n$ has
  size at most $d_n$.

  The case $n = 1$ is vacuously true, since $d_1 = 1$, and there is no
  infinite thin set for a 1-coloring of $\omega$.

  Assume that the property holds up to $n-1$.  Let
  $f_n(x_0, \dots, x_{n-1})$ be the largeness graph
  $\Gcal = (\{0, \dots, n-1\}, E)$ such that for $i < n-1$,
  $\{i,i+1\} \in E$ iff $[x_i, x_{i+1}]$ is $g$-large, and for
  $i+1 < j < n$, $\{i,j\} \in E$ iff $[x_i, x_{i+1}]$ is
  $g_{x_j}$-small.

  We now prove that any infinite $f$-thin set computes a function
  dominating~$g$. Let $H$ be an infinite $f$-thin set for some
  largeness graph $\Gcal = (\{0, \dots, n-1\}, E)$. We have two cases.

  Case 1: $\Gcal$ is not a packed largeness graph. There is some
  $i < n-1$ such that $\{i, i+1\} \in E$. There must be some
  $x_0 < \dots < x_i \in H$ such that $f_{i+1}(x_0, \dots, x_i)$ is
  the largeness subgraph of $\Gcal$ of size $i+1$ induced by the
  vertices $\{0, 1, \dots, i\}$, otherwise $H$ is $f_{i+1}$-thin, and
  by induction hypothesis, $H$ computes a function dominating~$g$ and
  we are done. Let $t$ be large enough such that $[x_i, t]$ is
  $g$-large.  There must be some
  $x_{i+1} < \dots < x_{n-1} \in H - \{0, \dots, t\}$ such that
  $f_{n-i-1}(x_{i+1}, \dots, x_{n-1})$ is the largeness subgraph of
  $\Gcal$ induced by the vertices $\{i+1, \dots, n-1\}$, otherwise
  $H - \{0, \dots, t\}$ is $f_{n-i-1}$-thin, and by induction
  hypothesis, $H$ computes a function dominating~$g$. In particular,
  $[x_i, x_{i+1}]$ is $g$-large, so $f(x_0, \dots, x_{n-1}) =
  \Gcal$. This contradicts the fact that $H$ is $f_n$-thin for color
  $\Gcal$.

  Case 2: $\Gcal$ is a packed largeness graph.  Note that the set
 $$
 W = \{ \{x_0, x_1, \dots, x_{n-1}\} \in [H]^n : f_n(x_0, \dots,
 x_{n-1}) \sim \Gcal \}
 $$
 is $H$-c.e. Indeed, this requires only to check that for
 $i+1 < j < n$, $\{i,j\} \in E$ iff $[x_i, x_{i+1}]$ is
 $g_{x_j}$-small, which is decidable.  Since $H$ is $f_n$-thin for
 color $\Gcal$, for every $\{x_0, \dots, x_{n-1}\} \in W$, one of
 $[x_i, x_{i+1}]$ is $g$-large. Therefore, it suffices to prove that
 for every $t \in \omega$, there is such a tuple with
 $t < \min\{x_0, \dots, x_{n-1}\}$. Indeed, if so, the function
 $h : \omega \to \omega$ which on $t$ $H$-computably searches for such
 a tuple $\{x_0, \dots, x_{n-1}\} \in W$ and outputs $x_{n-1}$ is a
 function dominating~$g$.
 
 By Lemma~\ref{lem:equivalence-normal-largeness-graph}, there is a
 normal largeness graph
 $\Gcal_1 = (\{0, \dots, n-1\},E_1) \sim \Gcal$. For every
 $t \in \omega$, there is some
 $x_0 < \dots < x_{n-2} \in H - \{0, \dots, t\}$ such that
 $f_{n-1}(x_0, \dots, x_{n-2})$ is the largeness subgraph of $\Gcal_1$
 of size $n-1$ induced by the vertices $\{0, \dots, n-2\}$, otherwise
 $H - \{0, \dots, t\}$ is $f_{n-1}$-thin, and by induction hypothesis,
 $H$ computes a function dominating~$g$.  Let $x_{n-1} \in H$ be
 sufficiently large so that $[x_{n-2}, x_{n-1}]$ is $g$-large. Then
 $f_n(x_0, \dots, x_{n-1}) = \Gcal_1$, and therefore
 $\{x_0, \dots, x_{n-1}\} \in W$.  This completes the proof of
 Theorem~\ref{thm:rt354-domination-left-ce}.
\end{proof}

\begin{corollary}
  For every $n \geq 2$ and $k \geq 1$, there is a computable instance
  of $\rt^{n+k}_{<\infty, d_n-1}$ such that every solution computes
  $\emptyset^{(k)}$.
\end{corollary}
\begin{proof}

  $\emptyset^{(k)}$ is c.e.\ in $\emptyset^{(k-1)}$.  Let
  $g : \omega \to \omega$ be a $\emptyset^{(k-1)}$-computable modulus
  of $\emptyset^{(k)}$.  $g$ is left c.e.\ over $\emptyset^{(k-1)}$.   By
  Theorem~\ref{thm:rt354-domination-left-ce}, there is a
  $\emptyset^{(k)}$-computable function $f : [\omega]^n \to d_n$ such
  that every infinite $f$-thin set computes a function dominating $g$,
  hence computes $\emptyset^{(k)}$.  By Schoenfield's limit lemma,
  there is a computable function $h : [\omega]^{n+k} \to d_n$ such
  that for every $x_0 < \dots < x_{n-1} \in \omega$,
$$
f(x_0, \dots, x_{n-1}) = \lim_{x_n} \ldots \lim_{x_{n+k-1}} h(x_0,
\dots, x_{n+k-1})
$$
Every infinite $h$-thin set is $f$-thin, and therefore computes
$\emptyset^{(k)}$.
\end{proof}

By a relativization of the proof of
Theorem~\ref{thm:rt354-domination-left-ce} and using more colors, one
can code every arithmetical set.

\begin{theorem}\label{thm:rt354-every-arithmetical-set}
  Let $A$ be an arithmetical set.  For every $n \geq 1$, there is a
  coloring $f : [\omega]^n \to k \cdot d_n$ such that every infinite
  set $H$ such that $|f[H]^n| < d_n$ computes $A$.
\end{theorem}
\begin{proof}
  Fix $n$ and $A$.  Since $A$ is arithmetical, it is $\Delta^0_{k+1}$
  for some $k \in \omega$.  Let $\Gamma$ be a functional such that for
  every set $X$, $\Gamma^X$ is a left-c.e.\ modulus of $X'$ relative
  to $X$, that is, for every function $g$ dominating $\Gamma^X$,
  $g \oplus X \geq_T X'$.

  Let $\Psi_n$ be the functional such that
  $\Psi_n^X(x_0, \dots, x_{n-1})$ is the largeness coloring of
  $(x_0, \dots, x_{n-1})$ defined by setting $\{x_i, x_{i+1}\}$ to be
  $\ell$ if $[x_i, x_{i+1}]$ is $\Gamma^X$-large, and
  $\{x_i, x_{i+1}\}$ is $s$ otherwise. Moreover, $\{x_i, x_{i+2}\}$
  has color $\ell$ if either $[x_i,x_{i+1}]$ is $\Gamma^X$-large, or
  $[x_i, x_{i+1}]$ is $x_{i+2}$-$\Gamma^X$-small, and
  $\{x_i, x_{i+2}\}$ has color $s$ otherwise.

  By a relativization of the proof of
  Theorem~\ref{thm:rt354-domination-left-ce}, for every
  $\Psi^X_n$-thin set $H$, $H \oplus X$ computes a function dominating
  $\Gamma^X$, hence computes $X'$. Let
$$
f_n(x_0, \dots, x_{n-1}) = \langle \Psi^\emptyset_n(x_0, \dots,
x_{n-1}), \dots, \Psi^{\emptyset^{(k-1)}}_n(x_0, \dots, x_{n-1})
\rangle
$$
See $f_n$ as an instance of $\rt^n_{<\infty, d_n-1}$. Let $H$ be an
infinite set such that $|f_n[H]^n| \leq d_n-1$. In particular, $H$ is
$\Psi^\emptyset_n$-thin, so $H \geq_T \emptyset'$. Moreover, $H$ is
$\Psi^{\emptyset'}_n$-thin, so
$H \oplus \emptyset' \geq_T \emptyset''$, hence
$H \geq_T \emptyset''$. By iterating the argument,
$H \geq_T \emptyset^{(k)}$, hence $H \geq_T A$. This completes the
proof.
\end{proof}

\section{The weakness of the thin set
  theorems}\label{sect:weakness-ts}

Wang~\cite{Wang2014Some} proved that $\rt^n_{<\infty,\ell}$ admits
strong cone avoidance whenever $\ell$ is at least the $n$th Schr\"oder
number, and asked whether this bound is optimal. In this section, we
answer negatively this question and prove that the exact bound
corresponds to Catalan numbers. We also prove the tightness of Dorais
et al.~\cite{Dorais2016uniform} by proving that $\rt^n_{<\infty,\ell}$
admits strong cone avoidance for non-arithmetical cones whenever
$\ell \geq 2^{n-1}$.

% Given some $k \in \omega$, we denote by $k^{<\infty}$ the set of all
% finite strings over $\{0, \dots, k-1\}$, and by $k^n$ the set of all
% finite strings over $\{0, \dots, k-1\}$ of length $n$. Given some
% $\sigma \in k^{<\omega}$ and a finite set
% $F \subseteq \{0, \dots, k-1\}$, we let
% $\setf_F(\sigma) = \{ i < |\sigma| : \sigma(i) \in F \}$. Whenever
% $F = \{c\}$, then we write $\setf_c(\sigma)$ for
% $\setf_{\{c\}}(\sigma)$.

\begin{definition}
  Let $\Dec_n$ be the set of all strictly decreasing non-empty
  sequences over $\{1, \dots, n-1\}$.  Given some $\sigma \in \Dec_n$,
  we let $\sigma^{+}$ be its last (smallest) element, and $\sigma^{-}$
  be the sequence truncated by its last element. If $|\sigma| = 1$,
  then $\sigma^{-}$ is the empty sequence $\epsilon$.  By convention,
  we also define $\varepsilon^{+} = n$.
\end{definition}

By a simple counting argument, $|\Dec_n| = 2^{n-1}-1$. Indeed, the
strictly decreasing non-empty sequences over $\{1, \dots, n-1\}$ are
in one-to-one correspondance with the non-empty subsets of
$\{1, \dots, n-1\}$.

We will use the set of decreasing sequences in the proof of
Theorem~\ref{thm:grtn-strong-cone-avoidance-direct}. See the
explanations before
Definition~\ref{def:grtn-strong-cone-avoidance-direct-index}.  We also
need the following technical definition which will be used in
Lemma~\ref{lem:grtn-thinout-compatible}.

\begin{definition}\label{def:inductive-size-colors}
  Fix $n \in \omega$ and a vector
  $\vec{\ell} = \langle \ell_1, \ell_2, \dots, \ell_n \rangle$ of
  natural numbers.  Given some $\sigma = n_0n_1\dots n_s \in \Dec_n$,
  let
  $\vec{\ell}(n,\sigma) = \ell_{n-n_0} \cdot \prod_{i = 1}^s
  \ell_{n_{i-1}-n_i}$
  % \times  \dots \times \ell_{n_{s-1} - n_s}$.
  By convention,
  $\vec{\ell}(n, \varepsilon) = 1$.
\end{definition}

% We can prove inductively that for every $\sigma \in \Dec_n$ and
% $S \in C_\sigma$, $|S| = \vec{\ell}(n,\sigma)$.

\begin{theorem}\label{thm:grtn-strong-cone-avoidance-direct}
  Fix $n \geq 1$, and let $\Mcal$ be a countable Scott set such that
$$
(\forall s \in \{1, \dots, n\})(\exists \ell_s)\Mcal \models
\rt^s_{<\infty, \ell_s}
$$
Let
$$
\ell = \ell_n + \sum_{\sigma \in \Dec_n} \vec{\ell}(n,\sigma) \cdot
\ell_{\sigma^{+}}
$$
For every $B \not \in \Mcal$ and every instance $f : [\omega]^n \to k$
of $\rt^n_{<\infty, \ell}$, there is a solution $G$ such that for
every $C \in \Mcal$, $B \not \leq_T G \oplus C$.
\end{theorem}
\begin{proof}
  Fix $n$, $\Mcal$, $B$, and $f$.

  Given two sets $A$ and $B$, we write $A \subseteq_n B$ for
  $A \subseteq B$ and $|A| \leq n$.  We identify an integer
  $k \in \omega$ with the set $\{0, \dots, k-1\}$.  We want to
  construct an infinite set $G$ such that $A \not \leq_T G \oplus C$
  for every $C \in \Mcal$, and $f[G]^n \subseteq_\ell k$.  Suppose
  there is no such set, otherwise we are done.  We are going to build
  our set $G$ by forcing.

  Let us illustrate the general idea in the case $n = 3$ with a
  function $f : [\omega]^3 \to k$. We write $\Pcal_s(X)$ for the
  collection of all finite subsets of $X$ of size $s$. Our goal is to
  build an infinite set $G$ such that $f[G]^3$ will use at most $\ell$
  colors.  For this, we will use a variant of Mathias forcing with
  conditions of the form $(F^K, X : K \in \Pcal_\ell(k))$. Here, we
  build simultaneously $|\Pcal_\ell(k)| = {k \choose \ell}$ many
  solutions. For each $K \in \Pcal_\ell(k)$, $F^K$ represents a finite
  stem of a solution $G^K$ such that $f[G^K]^3 \subseteq K$. We will
  ensure that at least one of $G^K : K \in \Pcal_\ell(k)$ will be
  infinite and be cone avoiding. The set $X$ is a shared reservoir
  from which all the future elements of $F^K$ will come. During the
  construction, the reservoir $X$ will become more and more
  restrictive, so that $\Pi^0_1$ facts about the constructed solution
  can be forced.

  We therefore require a condition
  $c = (F^K, X : K \in \Pcal_\ell(k))$ to satisfy the following
  property:
  \begin{quote}
    (1): For every $K \in \Pcal_\ell(k)$, $f[F^K]^3 \subseteq K$.
  \end{quote}
  However, this property is not enough to ensure that the stems are
  extendible.  Indeed, given a finite set $E \subseteq X$ satisfying
  again $f[E]^3 \subseteq K$, it may not be the case that
  $f[F^K \cup E]^3 \subseteq K$. A bad case is when there is some
  $x \in F^K$ such that for every $y, z \in X$,
  $f(x, y, z) \not \in K$. Another bad case is when for some
  $x, y \in F^K$, for every $z \in X$, $f(x, y, z) \not \in K$. We
  therefore need to strengthen the property (1). We would therefore
  want add the following properties:
  \begin{quote}
    (2.1): For every $K \in \Pcal_\ell(k)$, every $x, y \in F^K$, $\lim_z f(x, y, z) \in K$.\\
    (2.2): For every $K \in \Pcal_\ell(k)$, every $x \in F^K$,
    $\lim_y \lim_z f(x, y, z) \in K$.
  \end{quote}
  Properties (2.1) and (2.2) are in charge of propagating property
  (1), in that if $E \subseteq X$ is such that $f[E]^3 \subseteq K$,
  then $f[F^K \cup E]^3 \subseteq K$. There is however an issue: there
  is no reason to believe the function $f$ admits a limit. By Ramsey's
  theorem, we know there is a restriction $Y$ of the reservoir $X$
  over which $f$ admits a limit, but we cannot ensure that
  $Y \in \Mcal$. We will therefore have to \qt{guess} every
  possibility of limiting behavior of the function $f$. The limiting
  behavior of the function $f$ can be specified by two functions
  $g_2 : [\omega]^2 \to k$ and $g_1 : \omega \to k$, which informally
  should satisfy the equations $g_2(x, y) = \lim_z f(x, y, z)$ and
  $g_1(x) = \lim_y \lim_z f(x, y, z)$. The equations become:
  \begin{quote}
    (2.1): For every $K \in \Pcal_\ell(k)$, every $x, y \in F^K$, $g_2(x, y) \in K$.\\
    (2.2): For every $K \in \Pcal_\ell(k)$, every $x \in F^K$,
    $g_1(x) \in K$.
  \end{quote}
  Since we don't know the functions $g_2$ and $g_1$ ahead of time, we
  will need to try every possibility. Therefore, the notion of forcing
  becomes
  $(F^K_{g_2,g_1}, X : K \in \Pcal_\ell(k), g_2 : [\omega]^2 \to k,
  g_1 : \omega \to k)$.  A condition must satisfy the following
  properties for every $K \in \Pcal_\ell(k)$, every
  $g_2 : [\omega]^2 \to k$, and every $g_1 : \omega \to k$:
  \begin{quote}
    (1): $f[F^K_{g_2, g_1}]^3 \subseteq K$\\
    (2.1): For every $x, y \in F^K_{g_2, g_1}$, $g_2(x, y) \in K$.\\
    (2.2): For every $x \in F^K_{g_2,g_1}$, $g_1(x) \in K$.
  \end{quote}
  As explained, properties (2.1) and (2.2) are useful to propagate
  Property (1). However, we will encounter a similar issue of the
  propagation of Property (2.1): Suppose that there is some
  $x \in F^K_{g_2, g_1}$ such that for every $y \in X$,
  $g_2(x, y) \not \in K$. Then the property (2.1) will never be
  satisfied for $F^K_{g_2,g_1}$ for any set $E \subseteq X$. We need
  to also consider the limiting behavior of the function $g_2$. It is
  specified by a function $g_{2,1} : \omega \to k$. The notion of
  forcing becomes
  $(F^K_{g_2,g_{2,1},g_1}, X : K \in \Pcal_\ell(k), g_2 : [\omega]^2
  \to k, g_{2,1} : \omega \to k, g_1 : \omega \to k)$.  A condition
  must satisfy the following properties for every
  $K \in \Pcal_\ell(k)$, every $g_2 : [\omega]^2 \to k$, every
  $g_{2,1} : [\omega]^2 \to k$ and every $g_1 : \omega \to k$:
  \begin{quote}
    (1): $f[F^K_{g_2,g_{2,1}, g_1}]^3 \subseteq K$\\
    (2.1): For every $x, y \in F^K_{g_2,g_{2,1}, g_1}$, $g_2(x, y) \in K$.\\
    (2.1.1): For every $x \in F^K_{g_2,g_{2,1}, g_1}$, $g_{2,1}(x) \in K$.\\
    (2.2): For every $x \in F^K_{g_2,g_{2,1},g_1}$, $g_1(x) \in K$.\\
  \end{quote}
  Thus, Property (2.1.1) is necessary to propagate Property (2.1), and
  Properties (2.1) and (2.2) are necessary for Property (1).

  Actually, for technical reasons appearing in
  Lemma~\ref{lem:grtn-thinout-compatible}, given a function
  $f : [\omega]^3 \to k$, we will define its limit behavior $g_2$ on
  $a < b$ by applying $\rt^{3-2}_{<\infty, \ell_{3-2}}$ to the
  function $c \mapsto f(a, b, c)$ and let $g_2(a, b)$ be the resulting
  set of colors. Therefore, $g_2(a, b)$ will not be a single limit
  color, but a set of colors of size $\ell_{3-2}$. Thus $g_2$ has type
  $[\omega]^2 \to \Pcal_{\ell_{3-2}}(k)$. Similarly, $g_1$ will be
  defined on $a$ by applying $\rt^{3-1}_{<\infty, \ell_{3-1}}$ to the
  function $(b, c) \mapsto f(a, b, c)$ and let $g_1(a)$ be the limit
  set of colors of size $\ell_{3-1}$. So $g_1$ has type
  $\omega \to \Pcal_{\ell_{3-1}}$. Last, $g_{2,1}$ is now defining the
  limit behavior of the function
  $g_2 : [\omega] \to \Pcal_{\ell_{3-2}}(k)$. It will be defined on
  input $a$ by applying $\rt^{2-1}_{<\infty, \ell_{2-1}}$ to the
  function $b \mapsto g_2(a, b)$. Then, we get $\ell_{2-1}$ many
  values of $g_2$. However, the values of $g_2$ are also sets of
  colors of size $\ell_{3-2}$. Therefore, $g_{2,1}(a)$ will collect a
  set of colors of size $\ell_{3-2}\times \ell_{2-1}$. Thus $g_{2,1}$
  is of type $\omega \to \Pcal_{\ell_{3-2}\times
    \ell_{2-1}}(k)$. Notice that the number of colors corresponds to
  Definition~\ref{def:inductive-size-colors}. The notion of forcing
  becomes
  $(F^K_{g_2,g_{2,1},g_1}, X : K \in \Pcal_\ell(k), g_2 : [\omega]^2
  \to \Pcal_{\ell_1}(k), g_{2,1} : \omega \to \Pcal_{\ell_1}(k), g_1 :
  \omega \to \Pcal_{\ell_1}(k))$.  A condition must satisfy the
  following properties for every $K \in \Pcal_\ell(k)$, every
  $g_2 : [\omega]^2 \to \Pcal_{\ell_1}(k)$, every
  $g_{2,1} : [\omega]^2 \to \Pcal_{\ell_1}(k)$ and every
  $g_1 : \omega \to \Pcal_{\ell_1}(k)$:
  \begin{quote}
    (1): $f[F^K_{g_2,g_{2,1}, g_1}]^3 \subseteq K$\\
    (2.1): For every $x, y \in F^K_{g_2,g_{2,1}, g_1}$, $g_2(x, y) \subseteq K$.\\
    (2.1.1): For every $x \in F^K_{g_2,g_{2,1}, g_1}$, $g_{2,1}(x) \subseteq K$.\\
    (2.2): For every $x \in F^K_{g_2,g_{2,1},g_1}$, $g_1(x) \subseteq K$.\\
  \end{quote}

  Last, we can use compactness to the definition of a
  condition. Indeed, if for every
  $g_2 : [\omega]^2 \to \Pcal_{\ell_1}(k)$, every
  $g_{2,1} : [\omega]^2 \to \Pcal_{\ell_1}(k)$ and every
  $g_1 : \omega \to \Pcal_{\ell_1}(k)$, we can find a tuple
  $\langle F^K_{g_2,g_{2,1}, g_1} : K \in \Pcal_\ell(k) \rangle$
  satisfying properties (1), (2.1), (2.1.1) and (2.2), then we can
  find finitely many such tuples covering all the possible functions
  $g_2, g_{2,1}$ and $g_1$. Therefore, a condition becomes a tuple
  $(F^K_{g_2,g_{2,1},g_1}, X : K \in \Pcal_\ell(k), g_2 : [\{0, \dots,
  p-1\}]^2 \to \Pcal_{\ell_1}(k), g_{2,1} : \{0,\dots, p-1\} \to
  \Pcal_{\ell_1}(k), g_1 : \{0,\dots, p-1\} \to \Pcal_{\ell_1}(k))$
  satisfying properties (1), (2.1), (2.1.1) and (2.2).

  The exact computation of the size $\ell$ of the set $K$ appears in
  Case 2 of Lemma~\ref{lem:grtn-sca-direct-progress}, in order to
  force $\Pi^0_1$ facts.
  % This motivates the following definition.

\begin{definition}\label{def:grtn-strong-cone-avoidance-direct-index}
  Let $\Ib^{<\omega}$ be the set of all tuples
  $\vec{g} = \langle g_\sigma : \sigma \in \Dec_n \rangle$ such that
  for every $\sigma \in \Dec_n$, $g_\sigma$ is a function of type
  $[\{0, \dots, p-1\}]^{\sigma^{+}} \to
  \Pcal_{\vec{\ell}(n,\sigma)}(k)$ for some $p \in \omega$. We then
  let the \emph{height} of $\vec{g}$ be $\height(\vec{g}) = p$.  Let
  $\Ib^\omega$ be the set of all tuples
  $\vec{h} = \langle h_\sigma : \sigma \in \Dec_n \rangle$ such that
  for every $\sigma \in \Dec_n$, $h_\sigma$ is a function of type
  $[\omega]^{\sigma^{+}} \to \Pcal_{\vec{\ell}(n,\sigma)}(k)$.

  Given some
  $\vec{h} = \langle h_\sigma : \sigma \in \Dec_n \rangle \in
  \Ib^{<\omega} \cup \Ib^\omega$ and
  $\vec{g} = \langle g_\sigma : \sigma \in \Dec_n \rangle \in
  \Ib^{<\omega}$, we write $\vec{h} \leq \vec{g}$ if
  $g_\sigma \subseteq h_\sigma$ for each $\sigma \in \Dec_n$.

  An \emph{index set} is a finite set $I \subseteq \Ib^{<\omega}$ such
  that for every $\vec{h} \in \Ib^\omega$, there is a tuple
  $\vec{g} \in I$, such that $\vec{h} \leq \vec{g}$. Given two index
  sets $I, J$, we write $J \leq I$ if for every $\vec{h} \in J$, there
  is some $\vec{g} \in I$ such that $\vec{h} \leq \vec{g})$. The
  \emph{height} of an index $I$ is
  $\height(I) = \max \{ \height(\vec{g}) : \vec{g} \in I \}$.
\end{definition}

\begin{definition}\label{def:grtn-strong-cone-avoidance-direct-forcing}
  A \emph{condition} is a tuple
  $(F^K_{\vec{g}}, X : K \in \Pcal_\ell(k), \vec{g} \in I)$ such that,
  letting $g_\epsilon = \vec{x} \mapsto \{f(\vec{x})\}$,
  \begin{itemize}
  \item[(a)] $I$ is an index set
  \item[(b)] $g_\sigma(\vec{x}) \subseteq K$ for each
    $K \in \Pcal_\ell(k)$, $\sigma \in \Dec_n \cup \{\epsilon\}$,
    $\vec{g} \in I$ and $\vec{x} \in [F^K_{\vec{g}}]^{\sigma^{+}}$
  \item[(c)] $X \in \Mcal$ is an infinite set with $\min X > h(I)$
  \end{itemize}
\end{definition}
Note that the reservoir $X$ is shared with all the stems
$F^K_{\vec{g}}$. We refer to $\vec{g}$ as a \emph{branch} of the
condition $c$. Each branch can be seen as specifying
${k \choose \ell}$ simultaneous Mathias conditions
$(F^K_{\vec{g}}, X)$ for each $K \in \Pcal_\ell(k)$.  Also note that,
letting $\sigma = \epsilon$, we require that
$f[F^K_{\vec{g}}]^n \subseteq K$ for each $\vec{g} \in I$.

\begin{definition}
  A condition
  $d = (E^K_{\vec{h}}, Y : K \in \Pcal_\ell(k), \vec{h} \in J)$
  \emph{extends} a condition
  $c = (F^K_{\vec{g}}, X : K \in \Pcal_\ell(k), \vec{g} \in I)$
  (written $d \leq c$) if $J \leq I$, $Y \subseteq X$, and for each
  $K \in \Pcal_\ell(k)$ and $\vec{h} \in J$ and $\vec{g} \in I$ such
  that $\vec{h} \leq \vec{g}$, $F^K_{\vec{g}} \subseteq E^K_{\vec{h}}$
  and $E^K_{\vec{h}} \setminus F^K_{\vec{g}} \subseteq X$.
\end{definition}

% We say that $d$ is a \emph{simple extension} of $c$ if $d \leq c$
% and $I = J$.  Given a branch $\vec{g} \in I$, $d$ is a
% \emph{$\vec{g}$-extension} of $c$ if $d \leq c$ and
% $I - \vec{g} \subseteq J$. In other words, a $\vec{g}$-extension
% forks only the branch $\vec{g}$.

% In what follows, let $h_\epsilon : [\omega]^n \to \Pcal_1(k)$ be the
% function $\vec{x} \mapsto \{f(\vec{x})\}$ and $\epsilon^{+} = n$.

\begin{definition}\label{def:grtn-compatibility}
  Let $\vec{h} \in \Ib^{\omega}$ and $\vec{g} \in \Ib^{<\omega}$ be
  such that $\vec{h} \leq \vec{g}$. Let
  $h_\epsilon : [\omega]^n \to \Pcal_1(k)$.  A set
  $F > \height(\vec{g})$ is \emph{$(h_\epsilon, \vec{h})$-compatible
    with $\vec{g}$} if for every $\sigma \in \Dec_n$, letting
  $\tau = \sigma^{-}$, every $\vec{x} \in \dom g_\sigma$ and every
  $\vec{y} \in [F]^{\tau^{+} - \sigma^{+}}$, then
  $g_\sigma(\vec{x}) \supseteq h_\tau(\vec{x},\vec{y})$.
\end{definition}

The notion of $(h_\epsilon,\vec{h})$-compatibility has been designed
so that one can join two sets $F$ and $E$ satisfying property (b) of a
forcing condition, and obtain a set $F \cup E$ still satisfying
property (b), as proven in Lemma~\ref{lem:grtn-compatible-union}.

\begin{lemma}\label{lem:grtn-compatible-union}
  Let $\vec{h} \in \Ib^{\omega}$ and $\vec{g} \in \Ib^{<\omega}$ be
  such that $\vec{h} \leq \vec{g}$. Let $h_\epsilon = g_\epsilon$ be a
  function of type $[\omega]^n \to \Pcal_1(k)$.  Let $F$ and $E$ be
  two sets such that $F < \height(\vec{g}) < E$ and
  \begin{itemize}
  \item[(a)] $g_\sigma(\vec{x}) \subseteq K$ for each
    $\sigma \in \Dec_n \cup \{\epsilon\}$ and
    $\vec{x} \in [F]^{\sigma^{+}}$
  \item[(b)] $h_\sigma(\vec{x}) \subseteq K$ for each
    $\sigma \in \Dec_n \cup \{\epsilon\}$ and
    $\vec{x} \in [E]^{\sigma^{+}}$
  \item[(c)] $E$ is $(h_\epsilon,\vec{h})$-compatible with $\vec{g}$
  \end{itemize}
  Then $h_\sigma(\vec{x}) \subseteq K$ for each
  $\sigma \in \Dec_n \cup \{\epsilon\}$ and
  $\vec{x} \in [F \cup E]^{\sigma^{+}}$.
\end{lemma}
\begin{proof}
  Let $\sigma \in \Dec_n \cup \{\epsilon\}$ and $p =
  \sigma^{+}$. Recall that by convention, $\epsilon^{+} = n$.  We show
  that $h_\sigma(\vec{x}) \subseteq K$ for each
  $\vec{x} \in [F \cup E]^p$.  Let
  $\vec{x} = \{x_0, \dots, x_{p-1}\} \in [F \cup E]^p$, with
  $x_0 < \dots < x_{p-1}$.  We have three cases.  Case 1:
  $x_{p-1} \in F$. Then $\{x_0, \dots, x_{p-1}\} \in [F]^p$, and by
  (a),
  $h_\sigma(x_0, \dots, x_{p-1}) = g_\sigma(x_0, \dots, x_{p-1})
  \subseteq K$. Case 2: $x_0 \in F$. Then
  $\{x_0, \dots, x_{p-1}\} \in [E]^p$, and by (b),
  $h_\sigma(x_0, \dots, x_{p-1}) \subseteq K$. Case 3: there is some
  $i \in \{1, \dots, p-1\}$ such that $x_{i-1} \in F$ and $x_i \in
  E$. Let $\tau = \sigma^\frown i$. Since $0 < i < p$, then
  $\tau \in \Dec_n$. By (a),
  $g_\tau(x_0, \dots, x_{i-1}) \subseteq K$. Since
  $\sigma = \tau^{-}$, then by (c),
  $g_\tau(x_0, \dots, x_{i-1}) \supseteq h_\sigma(x_0, \dots,
  x_{p-1})$.  Hence $h_\sigma(x_0, \dots, x_{p-1}) \subseteq K$.
\end{proof}

\begin{definition}
  Let $\prec_L$ be a linearization of the prefix order $\prec$ on
  $\Dec_n$.  We have
  $\sigma_0 \prec_L \sigma_1 \prec_L \dots \prec_L
  \sigma_{2^{n-1}-2}$.  Given $p \in \omega$, let $\Tb_p$ be the set
  of all sequences
  $\langle g_{\sigma_0}, g_{\sigma_1}, \dots, g_{\sigma_s}\rangle$ for
  some $s < 2^{n-1}-1$, such that
  $g_{\sigma_i} : [p]^{\sigma_i^{+}} \to
  \Pcal_{\vec{\ell}(n,\sigma_i)}(k)$.  The empty sequence
  $\langle \rangle$ also belongs to $\Tb_p$.  The set $\Tb_p$ is
  naturally equipped with a partial order $\leq_{\Tb_p}$ corresponding
  to the prefix relation.

  A \emph{$\Tb_p$-tree} is a function $S$ whose domain is $\Tb_p$, and
  such that
  $S(\langle g_{\sigma_0}, g_{\sigma_1}, \dots, g_{\sigma_s}\rangle)$
  is a function
  $h_{\sigma_s} : [\omega]^{\sigma_s^{+}} \to
  \Pcal_{\vec{\ell}(n,\sigma_i)}(k)$. By convention,
  $S(\langle \rangle)$ is a function
  $h_\epsilon : [\omega]^n \to \Pcal_1(k)$.
\end{definition}

Note that the maximal sequences in $\Tb_p$ are precisely the
$\vec{g} \in \Ib^{<\omega}$ such that $\height(\vec{g}) = p$. In some
sense, $\Tb_p$ is the downward closure of such $\vec{g}$ under the
$\prec_L$ relation.  The following lemma justifies the combinatorial
design of the notion of forcing.

\begin{lemma}\label{lem:grtn-thinout-compatible}
  Let $S \in \Mcal$ be a $\Tb_p$-tree and $X \in \Mcal$ be an infinite
  set with $X > p$.  Then there is an infinite subset $Y \subseteq X$
  in $\Mcal$ and some $\vec{g} \in \Ib^{<\infty}$ with
  $\height(\vec{g}) = p$, such that, letting
  $h_\epsilon = S(\langle\rangle)$ and
  $\vec{h} = \langle S(\xi) : \xi \leq_{\Tb_p} \vec{g} \rangle$, $Y$
  is $(h_\epsilon, \vec{h})$-compatible with~$\vec{g}$.
\end{lemma}
\begin{proof}
  Fix
  $\sigma_0 \prec_L \sigma_1 \prec_L \dots \prec_L
  \sigma_{2^{n-1}-2}$.  We define inductively two sequences
  \begin{itemize}
  \item[(a)]
    $X = X_0 \supseteq X_1 \supseteq \dots \supseteq X_{2^{n-1}-1}$
    where $X_i \in \Mcal$ is an infinite set
  \item[(b)]
    $g_{\sigma_0}, g_{\sigma_1}, \dots, g_{\sigma_{2^{n-1}-2}}$ where
    $g_{\sigma_i} : [p]^{\sigma_i^{+}} \to
    \Pcal_{\vec{\ell}(n,\sigma_i)}(k)$
  \end{itemize}
  This induces a sequence
  $h_\epsilon, h_{\sigma_0}, h_{\sigma_1}, \dots,
  h_{\sigma_{2^{n-2}}}$ defined by $h_\epsilon = S(\langle \rangle)$
  and
  $h_{\sigma_i} = S(\langle g_{\sigma_0}, \dots,
  g_{\sigma_i}\rangle)$.  Note that $h_{\sigma_i} \in \Mcal$ is a
  function of type
  $[\omega]^{\sigma_i^{+}} \to \Pcal_{\vec{\ell}(n,\sigma_i)}(k)$.

  At step $i < 2^{n-1}-1$, we have already defined an infinite set
  $X_i \in \Mcal$ and the functions $g_{\sigma_j}$ and $h_{\sigma_j}$
  for every $j < i$.  Let $\tau = \sigma_i^{-}$, $a = \sigma_i^{+}$
  and $b = \tau^{+}$. Since $\prec_L$ is a linearization of $\prec$,
  the functions $g_\tau : [p]^b \to \Pcal_{\vec{\ell}(n,\tau)}(k)$ and
  $h_\tau : [\omega]^b \to \Pcal_{\vec{\ell}(n,\tau)}(k)$ are already
  defined.

  For each tuple $\{x_0, \dots, x_{a-1}\} \in [p]^a$, we can apply
  $\rt^{b-a}_{<\infty,\ell_{b-a}}$ to
  $x_a, \dots, x_{b-1} \mapsto h_\tau(x_0, \dots, x_{b-1})$ on the
  domain $X_i$ to obtain a set of colors
  $C \in \Pcal_{\ell_{b-a}}(\Pcal_{\vec{\ell}(n,\tau)}(k))$ and an
  infinite set $Y \subseteq X_i$ with $Y \in \Mcal$ such that
$$
(\forall \{x_a, \dots, x_{b-1}\} \in [Y]^{b-a})h_\tau(x_0, \dots,
x_{b-1}) \in C
$$
Let $g_{\sigma_i}(x_0, \dots, x_{a-1}) = \bigcup C$. Note that
$|C| = \ell_{b-a}$ and that each element of $C$ has size
$\vec{\ell}(n,\tau)$, so
$|\bigcup C| = \vec{\ell}(n,\tau) \cdot \ell_{b-a} =
\vec{\ell}(n,\sigma_i)$.  By applying the operation iteratively for
each tuple in $[p]^a$, we obtain an infinite set
$X_{i+1} \subseteq X_i$ in $\Mcal$ and a function
$g_{\sigma_i} : [p]^a \to \Pcal_{\vec{\ell}(n,\sigma_i)}(k)$ such that
for every $\{x_0, \dots, x_{a-1}\} \in [p]^a$ and
$\{x_a, \dots, x_{b-1}\} \in [X_{i+1}]^{b-a}$,
$$
g_{\sigma_i}(x_0, \dots, x_{a-1}) \supseteq h_\tau(x_0, \dots,
x_{b-1})
$$
We then go to the next step.  At the end of the construction, we
obtain an infinite set $X_{2^{n-1}-1}$ and some
$\vec{g} \in \Ib^{<\omega}$ satisfying the statement of our lemma.
\end{proof}

\begin{definition}
  Let $c = (F^K_{\vec{g}}, X : K \in \Pcal_\ell(k) , \vec{g} \in I)$
  be a condition and $\varphi(G, x)$ be a $\Delta^0_0$ formula.
  \begin{itemize}
  \item[(a)] $c \Vdash^K_{\vec{g}} (\exists x)\varphi(G, x)$ if there
    is some $x \in \omega$ such that $\varphi(F^K_{\vec{g}}, x)$ holds
  \item[(b)] $c \Vdash^K_{\vec{g}} (\forall x)\varphi(G, x)$ if for
    every $x \in \omega$, every $E \subseteq X$,
    $\varphi(F^K_{\vec{g}} \cup E, x)$ holds.
  \end{itemize}
\end{definition}

Note that the forcing relation for $\Pi^0_1$ formulas seems too strong
since no color restrain is imposed on the tuples over $E$.

\begin{definition}
  Let $c = (F^K_{\vec{g}}, X : K \in \Pcal_\ell(k) , \vec{g} \in I)$
  and $d = (E^K_{\vec{h}}, Y : K \in \Pcal_\ell(k), \vec{h} \in J)$ be
  conditions.  We say that $d$ is an \emph{$R$-extension} of $c$ for
  some $R \subseteq J$ if $R \neq \emptyset$ and $J - I \subseteq R$.
\end{definition}

\begin{lemma}\label{lem:grtn-sca-direct-progress}
  Let $c = (F^K_{\vec{g}}, X : K \in \Pcal_\ell(k), \vec{g} \in I)$ be
  a condition.  For every $\vec{g} \in I$, let
  $\langle e^K_{\vec{g}} : K \in \Pcal_\ell(k), \vec{g} \in I\rangle$
  be Turing indices.  Then there is a branch $\vec{g} \in I$ and an
  $R$-extension $d$ such that for every branch $\vec{h} \in R$ of $d$
  refining some branch $\vec{g}$ in $c$,
$$
d \Vdash^K_{\vec{h}} \Phi_e^{G \oplus C}(x) \uparrow
\hspace{10pt}\mbox{ or }\hspace{10pt} d \Vdash^K_{\vec{h}} \Phi_e^{G
  \oplus C}(x)\downarrow \neq A(x)
$$
for some $K \in \Pcal_\ell(k)$, some $x \in \omega$ and
$e = e^K_{\vec{g}}$.
\end{lemma}
\begin{proof}
  Fix
  $\sigma_0 \prec_L \sigma_1 \prec_L \dots \prec_L
  \sigma_{2^{n-1}-2}$.  Fix $p = \height(I)$, $x \in \omega$ and
  $v < 2$.

  Given some
  $\rho_{\sigma_{2^{n-1}-2}} : [p]^{\sigma^{+}_{2^{n-1}-2}} \to
  \Pcal_{\vec{\ell}(n,\sigma_{2^{n-1}-2})}(k)$ and some
  $h_\epsilon, h_{\sigma_0}, \dots, h_{\sigma_{2^{n-1}-3}}$ such that
  $h_\epsilon : [\omega]^n \to \Pcal_1(k)$ and
  $h_{\sigma_i} : [\omega]^{\sigma_i^{+}} \to
  \Pcal_{\vec{\ell}(n,\sigma_i)}(k)$, let
$$
\Ccal_{x,v}(\rho_{\sigma_{2^{n-1}-2}}, h_\epsilon, h_{\sigma_0},
\dots, h_{\sigma_{2^{n-1}-3}})
$$
be the $\Pi^0_1(X \oplus C \oplus \vec{h})$ class of all
$h_{\sigma_{2^{n-1}-2}} : [\omega]^{\sigma_{2^{n-1}-2}^{+}} \to
\Pcal_{\vec{\ell}(n,\sigma_{2^{n-1}-2})}(k)$ such that
$\rho_{\sigma_{2^{n-1}-2}} \subseteq h_{\sigma_{2^{n-1}-2}}$ and for
every $K \in \Pcal_\ell(k)$ and for every finite set $E \subseteq X$
such that $h_\sigma(\vec{x}) \subseteq K$ for each
$\sigma \in \Dec_n \cup \{\epsilon\}$ and
$\vec{x} \in [F^K_{\vec{\rho}} \cup E]^{\sigma^{+}}$,
$$
\Phi_{e^K_{\vec{g}}}^{(F^K_{\vec{g}} \cup E) \oplus C}(x) \uparrow
\mbox{ or } \Phi_{e^K_{\vec{g}}}^{(F^K_{\vec{g}} \cup E) \oplus C}(x)
\downarrow \neq v
$$
where $\vec{g} \in I$ is such that $\vec{h} \leq \vec{g}$.

Given some
$\rho_{\sigma_{2^{n-1}-3}} : [p]^{\sigma^{+}_{2^{n-1}-3}} \to
\Pcal_{\vec{\ell}(n,\sigma_{2^{n-1}-3})}(k)$ and some
$h_\epsilon, h_{\sigma_0}, \dots, h_{\sigma_{2^{n-1}-4}}$, let
$\Ccal_{x,v}(\rho_{\sigma_{2^{n-1}-3}}, h_\epsilon, h_{\sigma_0},
\dots, h_{\sigma_{2^{n-1}-4}})$ be the
$\Pi^0_1(X \oplus C \oplus \vec{h})$ class of all
$h_{\sigma_{2^{n-1}-3}} : [\omega]^{\sigma_{2^{n-1}-3}^{+}} \to
\Pcal_{\vec{\ell}(n,\sigma_{2^{n-1}-3})}(k)$ such that
$\rho_{\sigma_{2^{n-1}-3}} \subseteq h_{\sigma_{2^{n-1}-3}}$ and for
every
$\rho_{\sigma_{2^{n-1}-2}} : [p]^{\sigma^{+}_{2^{n-1}-2}} \to
\Pcal_{\vec{\ell}(n,\sigma_{2^{n-1}-2})}(k)$,
$$
\Ccal_{x,v}(\rho_{\sigma_{2^{n-1}-2}}, h_\epsilon, h_{\sigma_0},
\dots, h_{\sigma_{2^{n-1}-4}}, h_{\sigma_{2^{n-1}-3}}) \neq \emptyset
$$

And so on. Then we let $\Ccal_{x,v}$ be the $\Pi^0_1(X \oplus C)$
class of all $h_\epsilon : [\omega]^n \to \Pcal_1(k)$ such that for
every
$\rho_{\sigma_0} : [p]^{\sigma^{+}_0} \to
\Pcal_{\vec{\ell}(n,\sigma_0)}(k)$,
$\Ccal_{x,v}(\rho_{\sigma_0}, h_\epsilon) \neq \emptyset$.  Finally,
let $W = \{ (x, v) : \Ccal_{x, v} = \emptyset \}$.  Note that $W$ is
an $X \oplus C$-c.e. set. We have three cases.
\begin{itemize}
\item Case 1: There is some $x \in \omega$ such that
  $(x, 1-A(x)) \in W$. By definition, $\Ccal_{x, 1-A(x)} = \emptyset$.
  In particular, the function
  $h_\epsilon = \vec{x} \mapsto \{f(\vec{x})\}$ is not in
  $\Ccal_{x, 1-A(x)}$, so there is a
  $\rho_{\sigma_0} : [p]^{\sigma^{+}_0} \to
  \Pcal_{\vec{\ell}(n,\sigma_0)}(k)$ such that
  $\Ccal_{x,1-A(x)}(\rho_{\sigma_0}, h_\epsilon) = \emptyset$. By
  compactness, there is some $p_0 \in \omega$ such that for every
  $h_{\sigma_0} : [p_0]^{\sigma_0^{+}} \to
  \Pcal_{\vec{\ell}(n,\sigma_0)}(k)$, there is a
  $\rho_{\sigma_1} : [p]^{\sigma^{+}_1} \to
  \Pcal_{\vec{\ell}(n,\sigma_1)}(k)$ such that
  $\Ccal_{x,1-A(x)}(\rho_{\sigma_1}, h_\epsilon, h_{\sigma_0}) =
  \emptyset$.  By iterating the reasoning and assuming that $p_0$ is
  large enough to be the same witness of compactness, we obtain a
  non-empty collection $R$ of $\vec{h} \in \Ib^{<\omega}$ with
  $\height(\vec{h}) = p_0$ satisfying the following two properties:
  First, letting
  $J = R \cup \{\vec{g} \in I : (\forall \vec{h} \in R) \vec{h} \not
  \leq \vec{g}\}$, the set $J$ is an index set. Second, for each
  $\vec{h} \in R$, there is some $K \in \Pcal_\ell(k)$ and some finite
  set $E_{\vec{h}} \subseteq X$ such that
  $h_\sigma(\vec{x}) \subseteq K$ for each
  $\sigma \in \Dec_n \cup \{\epsilon\}$ and
  $\vec{x} \in [F^K_{\vec{\rho}} \cup E_{\vec{h}}]^{\sigma^{+}}$, and
$$
\Phi_{e^K_{\vec{g}}}^{(F^K_{\vec{g}} \cup E) \oplus C}(x) \downarrow =
1-A(x)
$$
where $\vec{g} \in I$ is such that $\vec{h} \leq \vec{g}$.  Define the
$R$-extension
$d = (H^K_{\vec{h}}, X - \{0, \dots, p_0\} : K \in \Pcal_\ell(k),
\vec{h} \in J)$ of $c$ by setting
$H^K_{\vec{h}} = F^K_{\vec{g}} \cup E_{\vec{h}}$ if $\vec{h} \in R$
and $\vec{g} \in I$ is such that $\vec{h} \leq \vec{g}$ and
$K_{\vec{g}} = K$.  Otherwise, set $H^K_{\vec{h}} = F^K_{\vec{g}}$
where $\vec{g} \in I$ is such that $\vec{h} \leq \vec{g}$.  For every
branch $\vec{h} \in R$ of $d$ refining $\vec{g}$, letting
$K = K_{\vec{h}}$ and $e = e^K_{\vec{g}}$,
$d \Vdash^K_{\vec{h}} \Phi_{e^K_{\vec{g}}}^{G \oplus C}(x)\downarrow
\neq A(x)$.

\item Case 2: There is some $x \in \omega$ such that
  $(x, 0), (x, 1) \not \in W$.  In particular $(x, A(x)) \not \in W$,
  so $\Ccal_{x, A(x)} \neq \emptyset$.  Since $\Mcal \models \wkl$,
  there is some $\Tb_p$-tree $S \in \Mcal$ such that for every maximal
  sequence
  $\vec{\rho} = \langle \rho_{\sigma_0}, \dots,
  \rho_{\sigma_{2^{n-1}-2}}\rangle \in \Tb_p$, letting
  $h_\epsilon = S(\langle \rangle)$ and, for each $i < 2^{n-1}-1$, letting
  $h_{\sigma_i} = S(\langle\rho_{\sigma_0}, \dots, \rho_{\sigma_i}
  \rangle)$, for every $K \in \Pcal_\ell(k)$ and for every finite set
  $E \subseteq X$ such that $h_\sigma(\vec{x}) \subseteq K$ for each
  $\sigma \in \Dec_n \cup \{\epsilon\}$ and
  $\vec{x} \in [F^K_{\vec{\rho}} \cup E]^{\sigma^{+}}$,
$$
\Phi_{e^K_{\vec{g}}}^{(F^K_{\vec{g}} \cup E) \oplus C}(x) \uparrow
\mbox{ or } \Phi_{e^K_{\vec{g}}}^{(F^K_{\vec{g}} \cup E) \oplus C}(x)
\downarrow \neq A(x)
$$
where $\vec{g} \in I$ is such that $\vec{h} \leq \vec{g}$.  By
Lemma~\ref{lem:grtn-thinout-compatible}, there is an infinite set
$Y \subseteq X$ in $\Mcal$, and some $\vec{\rho} \in \Ib^{<\infty}$
with $\height(\vec{g}) = p$, such that, letting
$h_\epsilon = S(\langle\rangle)$ and
$\vec{h} = \langle S(\xi) : \xi \leq_{\Tb_p} \vec{g} \rangle$, $Y$ is
$(h_\epsilon, \vec{h})$-compatible with~$\vec{\rho}$.  Let
$\vec{g} \in I$ be such that $\vec{\rho} \leq \vec{g}$. In particular,
$Y$ is $(h_\epsilon, \vec{h})$-compatible with $\vec{g}$.  Since for
each $\sigma \in \Dec_n \cup \{\epsilon\}$,
$\Mcal \models \rt^{\sigma^{+}}_{<\infty, \ell_{\sigma^{+}}}$, then by
an iterative process, we obtain an infinite set $Y_1 \subseteq Y$ in
$\Mcal$ and for each $\sigma \in \Dec_n \cup \{\epsilon\}$ some set of
colors
$C_\sigma \in
\Pcal_{\ell_{\sigma^{+}}}(\Pcal_{\vec{\ell}(n,\sigma)}(k))$ such that
$h_\sigma[Y_1]^{\sigma^{+}} \subseteq C_\sigma$. In particular, for
$\sigma \in \Dec_n$, $|C_\sigma| = \ell_{\sigma^{+}}$ and each element
of $C_\sigma$ has size $\vec{\ell}(n,\sigma)$, so
$|\bigcup C_\sigma| = \vec{\ell}(n,\sigma) \times
\ell_{\sigma^{+}}$. Moreover,
$C_\epsilon \in \Pcal_{\ell_n}(\Pcal_{\vec{\ell}(n,\epsilon)}(k))$
with $\vec{\ell}(n,\epsilon) = 1$, so $|C_\epsilon| = \ell_n$ and
$|\bigcup C_\epsilon| = \ell_n$. It follows that
\begin{eqnarray*}
  |\bigcup_{\sigma \in \Dec_n \cup \{\epsilon\}} \bigcup C_\sigma| &\leq & 
                                                                           |\bigcup C_\epsilon| + \sum_{\sigma \in \Dec_n} |\bigcup C_\sigma|| \\
                                                                   & \leq & \ell_n + \sum_{\sigma \in \Dec_n} \vec{\ell}(n,\sigma) \times \ell_{\sigma^{+}}	
\end{eqnarray*}
Therefore, there is some $K \in \Pcal_\ell(k)$ such that
$K \supseteq \bigcup_{\sigma \in \Dec_n \cup \{\epsilon\}} \bigcup
C_\sigma$.  By definition of a condition,
$g_\sigma(\vec{x}) \subseteq K$ for each
$\sigma \in \Dec_n \cup \{\epsilon\}$ and
$\vec{x} \in [F^K_{\vec{g}}]^{\sigma^{+}}$.  By choice of $K$,
$h_\sigma(\vec{x}) \subseteq K$ for each
$\sigma \in \Dec_n \cup \{\epsilon\}$ and
$\vec{x} \in [Y]^{\sigma^{+}}$.  By choice of $Y$, $Y$ is
$(h_\epsilon, \vec{h})$-compatible with $\vec{g}$.  Therefore, by
Lemma~\ref{lem:grtn-compatible-union}, $h_\sigma(\vec{x}) \subseteq K$
for each $\sigma \in \Dec_n \cup \{\epsilon\}$ and
$\vec{x} \in [F^K_{\vec{g}} \cup Y]^{\sigma^{+}}$.  The condition
$d = (F^K_{\vec{g}}, Y : K \in \Pcal_\ell(k), \vec{g} \in I)$ is a
$\{\vec{g}\}$-extension of $c$ such that
$d^K_{\vec{g}} \Vdash \Phi_e^{G \oplus C}(x)\uparrow$, where
$e = e^K_{\vec{g}}$.
	
\item Case 3: None of the above cases hold. In this case, we can
  $X \oplus C$-compute the set $A$, contradicting our assumption.
\end{itemize}
\end{proof}

For the simplicity of notation, given some $\vec{h} \in \Ib^\omega$
and a condition $c$ with index set $I$, we might write
$\Vdash^K_{\vec{h}}$ for $\Vdash^K_{\vec{g}}$ where $\vec{g}$ is the
unique branch in $I$ such that $\vec{h} \leq \vec{g}$.

Let $\Fcal$ be a sufficiently generic filter for this notion of
forcing.  By Lemma~\ref{lem:grtn-sca-direct-progress}, there is some
$\vec{h} \in \Ib^\omega$ such that for every tuple of indices
$\langle e_K \in \omega : K \in \Pcal_\ell(k) \rangle$,
\begin{equation}
  c \Vdash^K_{\vec{h}} \Phi_{e_K}^{G \oplus C}(x) \uparrow \hspace{10pt}\mbox{ or }\hspace{10pt} c \Vdash^K_{\vec{h}} \Phi_{e_K}^{G \oplus C}(x)\downarrow \neq A(x)
\end{equation}
for some $c \in \Fcal$, $K \in \Pcal_\ell(k)$ and some $x \in
\omega$. We claim that there is some $K \in \Pcal_\ell(k)$ such that
for every index $e \in \omega$,
\begin{equation}
  c \Vdash^K_{\vec{h}} \Phi_e^{G \oplus C}(x) \uparrow \hspace{10pt}\mbox{ or }\hspace{10pt} c \Vdash^K_{\vec{h}} \Phi_e^{G \oplus C}(x)\downarrow \neq A(x)
\end{equation}
for some $c \in \Fcal$ and some $x \in \omega$.  Indeed, suppose
not. Then for every $K \in \Pcal_\ell(k)$, there is some $e_K$ such
that for every $c \in \Fcal$ and $x \in \omega$, the equation (2) does
not hold. Then this contradicts the equation (1) for the tuple
$\langle e_K \in \omega : K \in \Pcal_\ell(k) \rangle$.

In what follows, we fix $\Fcal$, $\vec{h}$ and $K$ such that the
equation (2) holds.  Let
$$
G = \bigcup \{ F^K_{\vec{g}} : (F^K_{\vec{g}}, X : K \in \Pcal_\ell(k)
, \vec{g} \in I) \in \Fcal, \vec{h} \leq \vec{g} \}
$$

\begin{lemma}
  The set $G$ is infinite.
\end{lemma}
\begin{proof}
  Let $t \in \omega$.  Let $\Phi^{G \oplus C}_e$ be the Turing
  functional which on input $x$ searches for some $y \in G$ such that
  $y > t$. If found, the program halts and output 1. Otherwise it
  diverges.  Let
  $c = (F^K_{\vec{g}}, X : K \in \Pcal_\ell(k) , \vec{g} \in I) \in
  \Fcal$ and $x$ be such that
$$
c \Vdash^K_{\vec{h}} \Phi_e^{G \oplus C}(x) \uparrow
\hspace{10pt}\mbox{ or }\hspace{10pt} c \Vdash^K_{\vec{h}} \Phi_e^{G
  \oplus C}(x)\downarrow \neq A(x)
$$
 
Note that $c \not \Vdash^K_{\vec{h}} \Phi^{G \oplus C}_e(x) \uparrow$
since the reservoir $X$ is infinite. It follows that
$c \Vdash^K_{\vec{h}} \Phi_e^{G \oplus C}(x)\downarrow \neq A(x)$.
Unfolding the definition of the forcing relation,
$\Phi_e^{F^K_{\vec{g}} \oplus C}(x)\downarrow$ where $\vec{g} \in I$
is such that $\vec{h} \leq \vec{g}$. In other words,
$\max F^K_{\vec{g}} > t$.  Since $F^K_{\vec{g}} \subseteq G$, there is
some $y \in G$ with $y > t$.
\end{proof}

By construction, $f[G]^n \subseteq K$, and by choice of $\vec{h}$,
$G \oplus C$ does not compute $A$.  This completes the proof of
Theorem~\ref{thm:grtn-strong-cone-avoidance-direct}.
\end{proof}

\begin{theorem}\label{thm:rtn-strong-cone-avoidance-arith}
  For every $n \geq 1$, $\rt^n_{<\infty, 2^{n-1}}$ admits strong cone
  avoidance for non-arithmetical cones.
\end{theorem}
\begin{proof}
  Fix a set $C$, a set $A$ which is not arithmetical in $C$, and a
  coloring $f : [\omega]^n \to k$.  By Jockusch and
  Soare~\cite{Jockusch197201}, every computable instance of $\wkl$ has
  a low solution, and by Jockusch~\cite{Jockusch1972Ramseys}, every
  computable instance of $\rt$ has an arithmetical
  solution. Therefore, there is a countable $\omega$-model $\Mcal$ of
  $\wkl + \rt$ such that $C \in \Mcal$ and containing only sets
  arithmetical in $C$. In particular, $A$ is not arithmetical in any
  element of $\Mcal$.  By
  Theorem~\ref{thm:grtn-strong-cone-avoidance-direct}, letting
  \begin{align*}
    \ell & = 1 + \sum_{\sigma \in \Dec_n} 1 = 2^{n-1}
  \end{align*}
  there is an infinite set $G \subseteq \omega$ such that
  $f[G]^n \subseteq_{\ell} k$, and such that for every $C \in \Mcal$,
  $A \not \leq_T G \oplus C$.  This completes the proof of
  Theorem~\ref{thm:rtn-strong-cone-avoidance-arith}.
\end{proof}

Let $\ell_1, \ell_2, \dots$ be the sequence inductively defined by
$\ell_1 = 1$, and
$$
\ell_{n+1} = \ell_n + \sum_{\sigma \in \Dec_n} \vec{\ell}(n,\sigma)
\cdot \ell_{\sigma^{+}}
$$

\begin{lemma}\label{lem:ell-is-catalan}
  For every $n \geq 1$,
$$
\ell_{n+1} = \sum_{i=0}^{n-1} \ell_{i+1} \ell_{n-i}
$$
\end{lemma}
\begin{proof}
  By induction over $n$.  Assume
  $\ell_{i+1} = \sum_{i=0}^{i-1} \ell_{i+1} \ell_{i-1}$ for every
  $i \in \{1, \dots, n-1\}$.  By definition,
  $\ell_{n+1} = \ell_n + \sum_{\sigma \in \Dec_n} \vec{\ell}(n,\sigma)
  \cdot \ell_{\sigma^{+}}$.

  Fix $i \in \{1, \dots, n-1\}$. The strings $\sigma \in \Dec_n$ such
  that $\sigma(0) = i$ are precisely the strings of the form
  $i^\frown\tau$ for some $\tau \in \Dec_i \cup \{\epsilon\}$.
  Therefore,

$$
\begin{array}{rl}
  \sum_{\sigma \in \Dec_n, \sigma(0) = i} \vec{\ell}(n,\sigma) \cdot \ell_{\sigma^{+}}
  & = \sum_{\tau \in \Dec_i \cup \{\epsilon\}} \vec{\ell}(n,i^\frown \tau) \cdot \ell_{(i^\frown\tau)^{+}}\\
  & = \vec{\ell}(n, i)\ell_i + \sum_{\tau \in \Dec_i} \ell_{n-i} \cdot \vec{\ell}(i,\tau) \cdot \ell_{\tau^{+}}\\
  &= \ell_{n-i}(\ell_i + \sum_{\tau \in \Dec_i} \vec{\ell}(i,\tau) \cdot \ell_{\tau^{+}})\\
  &= \ell_{n-i} \cdot \ell_{i+1}
\end{array}
$$
Therefore,
$$
\ell_{n+1} = \ell_n + \sum_{\sigma \in \Dec_n} \vec{\ell}(n,\sigma)
\cdot \ell_{\sigma^{+}} = \ell_n + \sum_{i=1}^{n-1}
\ell_{n-i}\ell_{i+1} = \sum_{i=0}^{n-1} \ell_{i+1} \ell_{n-i}
$$
This completes the proof of Lemma~\ref{lem:ell-is-catalan}.
\end{proof}

Recall that $d_0, d_1, \dots$ denotes the Catalan sequence,
inductively defined by $d_0 = 1$ and
$$
d_{n+1} = \sum_{i=0}^n d_i d_{n-i}
$$

\begin{corollary}\label{cor:dn-is-ellnplus}
  For every $n \geq 0$, $d_n = \ell_{n+1}$.
\end{corollary}
\begin{proof}
  Immediate by Lemma~\ref{lem:ell-is-catalan}.
\end{proof}

\begin{theorem}\label{thm:rtn-strong-cone-avoidance-real}
  For every $n \geq 1$, $\rt^n_{<\infty, d_n}$ admits strong cone
  avoidance.
\end{theorem}
\begin{proof}
  By induction over $n \geq 1$.  Fix a set $C$, a set
  $A \not \leq_T C$, and a coloring $f : [\omega]^n \to k$.  By
  Jocksuch and Soare~\cite{Jockusch197201}, every $C$-computable
  instance of $\wkl$ has a solution $P$ such that
  $A \not \leq_T P \oplus C$. By induction hypothesis, we can build a
  countable $\omega$-model $\Mcal$ of
  $\wkl \bigwedge_{s \in \{1, \dots, n\}} \rt^s_{<\infty, d_{s-1}}$
  such that $C \in \Mcal$ and $A \not \in \Mcal$. By
  Corollary~\ref{cor:dn-is-ellnplus},
  $\Mcal \models \bigwedge_{s \in \{1, \dots, n\}} \rt^s_{<\infty,
    \ell_s}$.  By Theorem~\ref{thm:grtn-strong-cone-avoidance-direct},
  there is an infinite set $G \subseteq \omega$ such that
  $f[G]^n \subseteq_{\ell_{n+1}} k$, and such that for every
  $C \in \Mcal$, $A \not \leq_T G \oplus C$. By
  Corollary~\ref{cor:dn-is-ellnplus}, $\ell_{n+1} = d_n$. This
  completes the proof of
  Theorem~\ref{thm:rtn-strong-cone-avoidance-real}.
\end{proof}

\begin{corollary}\label{cor:rtn-cone-avoidance-real}
  For every $n \geq 1$, $\rt^{n+1}_{<\infty, d_n}$ admits cone
  avoidance.
\end{corollary}
\begin{proof}
  Immediate by Theorem~\ref{thm:rtn-strong-cone-avoidance-real} and
  Theorem~\ref{thm:bridge-strong-to-non-strong}.
\end{proof}

\section{The GAP principle}\label{sect:gap-principle}

As explained in Section~\ref{sect:strength-ts}, a candidate function
to improve the lower bound on the strength of the thin set theorem for
3-tuples was
$$
f(a, b, c) = \langle gap(a,b), gap(b,c), gap(a,c) \rangle
$$
where $gap(a,b) = \ell$ if $[a,b]$ is $g$-large, and $gap(a,b) = s$
otherwise.  In this section, we prove that there always exists an
infinite set $H \subseteq \omega$ which avoids the color
$\langle s, s, \ell \rangle$ and which does not compute the halting
set. We define the corresponding problem $\gap$, and study its reverse
mathematical strength.

\begin{definition}
  A set $H$ is \emph{$g$-transitive} if for every $x < y < z \in H$
  such that $[x, y]$ and $[y, z]$ are $g$-small, then $[x, z]$ is
  $g$-small.
\end{definition}

The notion of $g$-transitivity exactly says that the color
$\langle s, s, \ell \rangle$ is avoided for the previously defined
function $f$.

\begin{statement}
  $\gap$ is the statement \qt{For every increasing function
    $g : \omega \to \omega$, there is an infinite $g$-transitive set
    $H$.}  $\dgap$ is the statement \qt{For every $\Delta^0_2$
    increasing function $g : \omega \to \omega$, there is an infinite
    $g$-transitive set $H$.}
\end{statement}

The main motivation of the $\gap$ principle is the study of strong
cone avoidance of $\rt^3_{5,4}$.  We start by proving that $\gap$
follows from a stable version of the Erd\H{o}s-Moser theorem, which is
already known to admit strong cone avoidance.

\begin{definition}[Erd\H{o}s-Moser theorem] 
  A tournament $T$ on a domain $D \subseteq \N$ is an irreflexive
  binary relation on~$D$ such that for all $x,y \in D$ with
  $x \not= y$, exactly one of $T(x,y)$ or $T(y,x)$ holds. A tournament
  $T$ is \emph{transitive} if the corresponding relation~$T$ is
  transitive in the usual sense. A tournament $T$ is \emph{stable} if
  $(\forall x \in D)[(\forall^{\infty} s) T(x,s) \vee
  (\forall^{\infty} s) T(s, x)]$.  $\emo$ is the statement ``Every
  infinite tournament $T$ has an infinite transitive subtournament.''
  $\semo$ is the restriction of $\emo$ to stable tournaments.
\end{definition}

\begin{definition} %added cholak
  ${\sf{P}} \leq_{sc} \sf{Q}$ iff, for every instance of $\s{P}$,
  $I_{\s{P}}$, there is an instance of $\s{Q}$, $I_{\s{Q}}$, such that
  for every solution of $I_{\s{Q}}$, $S_{\s{Q}}$, computes a solution
  of $I_{\s{P}}$, $S_{\s{P}}$.
\end{definition}

\begin{theorem}
  $\gap \leq_{sc} \semo$.
\end{theorem}
\begin{proof}
  Let $g : \omega \to \omega$ be a function.  Set $T(x, y)$ to hold if
  $x < y$ and $[x, y]$ is $g$-small, or $y \leq x$ and $[x, y]$ is
  $g$-large.  Note that $T$ is stable.  Let $H$ be an infinite
  $T$-transitive subtournament. Then for every $x < y < z$ such that
  $[x, y]$ and $[y, z]$ are $g$-small, $T(x, y)$ and $T(y, z)$ holds.
  By $T$-transitivity of $H$, $T(x, z)$ holds, hence $[x, z]$ is
  $g$-small.
\end{proof}

\begin{corollary}
  $\gap$ admits strong cone avoidance.
\end{corollary}
\begin{proof}
  $\emo$ admits strong cone avoidance~\cite{PateyCombinatorial} and
  $\gap \leq_{sc} \emo$.
\end{proof}

\begin{corollary}
  $\rca + \dgap \nvdash \aca$.
\end{corollary}
\begin{proof}
  Build an $\omega$-model of $\rca + \dgap$ which does not contain the
  halting set.
\end{proof}

\begin{theorem}
  $\gap \leq_{sc} \rt^3_{5,4}$.
\end{theorem}
\begin{proof}
  Let $g : \omega \to \omega$ be a an increasing function.  Given
  $x < y \in \omega$, let $i(x,y) = 1$ if $[x, y]$ is $g$-large, and
  $i(x, y) = 0$ otherwise.  Let $f : [\omega]^3 \to 5$ be defined on
  $x < y < z$ by
  $f(x, y, z) = \langle i(x,y), i(y,zs), i(x,z)\rangle$. Note that $f$
  is a $5$-coloring, since the colors $\langle 1, 0, 0\rangle$,
  $\langle 0, 1, 0\rangle$, $\langle 1, 1, 0\rangle$ cannot occur.
  Let $H$ be an infinite set such that $f[H]^3$ avoids one color
  $c$. We have several cases.
  \begin{itemize}
  \item Case 1: $c = \langle 1, 1, 1 \rangle$. This case is
    impossible, since we can always pick three elements
    $x < y < z \in H$ sufficiently sparse so that $[x, y]$ and
    $[y, z]$ is $g$-large.
  \item Case 2: $c = \langle 0, 1, 1\rangle$. In this case, the set
    $H$ is only made of $g$-large intervals, and therefore is
    $g$-transitive. Indeed, suppose there is a $g$-small interval
    $[x, y]$ with $x < y \in H$. Then picking $z$ sufficiently far,
    $i(y, z) = 1$ and $i(x, z) = 1$, in which case
    $f(x, y, z) = \langle 0, 1, 1 \rangle$.
  \item Case 3: $c = \langle 1, 0, 1 \rangle$. Let $x < y \in H$ be
    such that $[x, y]$ is $g$-large, and let
    $G = \{ z \in H : z > y \}$. We claim that $G$ is made only of
    $g$-large intervals, hence is $g$-transitive. Indeed, suppose that
    $[u, v]$ is $g$-small for some $u < v \in G$. Then, since $[x, u]$
    is $g$-large, $f(x, u, v) = \langle 1, 0, 1\rangle$,
    contradiction.
  \item Case 4: $c = \langle 0, 0, 1 \rangle$. This means exactly that
    $H$ is $g$-transitive.
  \item Case 5: $c = \langle 0, 0, 0 \rangle$. Suppose
    $H = \{x_0 < x_1 < \dots \}$, and let
    $G = \{x_{2n} : n \in \omega \}$. We claim that $G$ is made only
    of $g$-large intervals, hence is $g$-transitive. Indeed, if
    $[x_{2n}, x_{2n+2}]$ were $g$-small, then so would be
    $[x_{2n}, x_{2n+1}]$ and $[x_{2n+1}, x_{2n+2}]$, in which case
    $f(x_{2n}, x_{2n+1}, x_{2n+2}) = \langle 0, 0, 0 \rangle$.
  \end{itemize}
  Note that case 3 is the only one which prevents the reduction from
  being uniform, and case 4 is the only one which prevents the
  reduction from constructing an infinite set on which all the
  intervals are $g$-large. In particular,
  $\gap \leq_{sW} \rt^3_{5,3}$.
\end{proof}

\begin{definition}
  A function $g$ is \emph{hyperimmune} if it is not dominated by any
  computable function.  An infinite set $H = \{ x_0 < x_1 < \dots \}$
  is \emph{hyperimmune} if the function $p_H$ defined by
  $p_H(n) = x_n$ is hyperimmune.
\end{definition}

\begin{lemma}\label{lem:domination-gives-transitivity}
  Let $g : \omega \to \omega$ be an increasing function. Every
  function dominating $g$ computes an infinite $g$-transitive set.
\end{lemma}
\begin{proof}
  Let $f$ be a function dominating $g$. Let
  $H = \{x_0 < x_1 < \dots \}$ be defined by $x_0 = 0$, and
  $x_{n+1} = f(x_n)$. Then every interval in $H$ is $f$-large, hence
  $g$-large, and in particular is $g$-transitive.
\end{proof}

\begin{definition}
  A function $f : \omega \to \omega$ is \emph{diagonally
    non-computable} relative to $X$ if $f(e) \neq \Phi^X_e(e)$ for
  every $e \in \omega$. $\dnc$ is the statement \qt{For every set $X$,
    there is a diagonally non-computable function relative to $X$.}
\end{definition}

We then say that a Turing degree is DNC if it contains a diagonally
non-computable function. The notion of DNC degree is very weak and not
being able to bound such a degree is a good measure of the
computability-theoretic weakness of a problem.

\begin{theorem}\label{thm:gap-preserve-hyperimmune-and-non-dnc}
  Let $g : \omega \to \omega$ be an increasing function and
  $f_0, f_1, \dots$ be a countable sequence of hyperimmune functions.
  Then there is an infinite $g$-transitive set $H$ of non-DNC degree
  such that $f_i$ is $H$-hyperimmune for every $i$.
\end{theorem}
\begin{proof}
  Suppose there is no computable infinite $g$-transitive set,
  otherwise we are done.  We will build the set $H$ using a variant of
  computable Mathias forcing.  A condition is a pair $(F, X)$ where
  $F$ is a finite $g$-transitive set, $X$ is an infinite, computable
  set such that $\max F < \min X$, and $[x, y]$ is $g$-large for every
  $x \in F$ and $y \in X$.  A condition $(E, Y)$ extends $(F, X)$ if
  $F \subseteq E$, $Y \subseteq X$ and $E \setminus F \subseteq X$.

  Every sufficiently generic filter for this notion of forcing yields
  an infinite $g$-transitive set $G$. Indeed, given a condition
  $(F, X)$, one can pick any $x \in X$, add it to $F$, and remove
  finitely many elements from $X$ to obtain an infinite set $Y$ such
  that $[x, y]$ is $g$-large for every $y \in Y$.

\begin{lemma}
  Given a Turing functional $\Phi_e$, some $i \in \omega$ and a
  condition $c$, there is an extension $d$ forcing either $\Phi_e^G$
  to be partial, or $\Phi_e^G$ not to dominate $f_i$.
\end{lemma}
\begin{proof}
  Fix a condition $(F, X)$.  Define the partial computable function
  $h : \omega \to \omega$ which on input $n$, searches for some finite
  sets $E_0 < E_1 < \dots < E_{k-1} \subseteq X$ such that
  $\Phi_e^{F \cup E_j}(n)\downarrow$ for each $j < k$, and
  $\Phi_e^{F \cup H}(n)\downarrow$, where $H = \{\max E_j : j <
  k\}$. If found,
  $h(n) = \max_j \{ \Phi_e^{F \cup E_j}(n), \Phi_e^{F \cup H}(n)
  \}$. Otherwise $h(n) \uparrow$. We have two cases.

  Case 1: $h$ is total. Since $f_i$ is hyperimmune, there is some $n$
  such that $h(n) < f_i(n)$.  Let
  $E_0 < E_1 < \dots < E_{k-1} \subseteq X$ witness that
  $h(n)\downarrow$, and let $Y$ be obtained from $X$ by removing
  finitely many elements so that $[x, y]$ is $g$-large for every
  $x \in E_i$ and $y \in Y$. Either there is some $E_j$ which is
  $g$-small, in which case, $(F \cup E_i, Y)$ is an extension forcing
  $\Phi_e^G(n)\downarrow < f_i(n)$, or for every $j < k$, $E_j$ is
  $g$-large. Then $H$ is made only of large intervals, so is
  $g$-transitive. $(F \cup H, Y)$ is then an extension forcing again
  $\Phi_e^G(n)\downarrow < f_i(n)$.

  Case 2: $h$ is partial, say $h(n)\uparrow$. Let $E_0 < E_1 < \dots$
  be a maximal computable sequence of finite sets such that
  $\Phi_e^{F \cup E_j}(n)\downarrow$ for each $j$. If the sequence is
  finite, then letting $Y = X \setminus \{0, \dots, \max E_j \}$,
  $(F, Y)$ forces $\Phi^G_e(n)\uparrow$.  If the sequence is infinite,
  then letting $Y = \{\max E_j : j \in \omega\}$, the condition
  $(F, Y)$ is an extension forcing again $\Phi^G_e(n)\uparrow$.
\end{proof}

By the previous lemma, $f_i$ is $G$-hyperimmune for every sufficiently
generic for this notion of forcing and every $i \in \omega$.

\begin{lemma}
  Given a Turing functional $\Phi_e$, and a condition $c$, there is an
  extension $d$ forcing either $\Phi_e^G$ to be partial, or
  $\Phi_e^G(n) = \Phi_n(n)$ for some $n \in \omega$.
\end{lemma}
\begin{proof}
  Fix a condition $(F, X)$.  Define a maximal computable sequence
  $E_0 < E_1 < \dots \subseteq X$ such that for every $n$,
  $\Phi_e^{F \cup E_n}(n)\downarrow$.  Suppose first that this
  sequence is finite, with maximum element $E_n$, then
  $(F, X - \{0, \dots, \max E_n\})$ is an extension forcing
  $\Phi_e^G(n+1)\uparrow$.  Suppose now that this sequence is
  infinite. Then, there must be infinitely many $n$ such that
  $\Phi^{F \cup E_n}_e(n) \downarrow = \Phi_n(n)$, otherwise we would
  compute a DNC function. Let
  $W = \{n : \Phi^{F \cup E_n}_e(n) \downarrow = \Phi_n(n) \}$.  If
  $E_n$ is $g$-small for some $n \in W$, then in particular
  $F \cup E_n$ is $g$-transitive. Let $d = (F \cup E_n, Y)$ where $Y$
  is obtained from $X$ by removing finitely many elements so that
  $[x, y]$ is $g$-large for every $x \in E_n$. The condition $d$ is an
  extension of $c$ forcing $\Phi_e^G(n) = \Phi_n(n)$.  If $E_n$ is
  $g$-large for every $n \in W$, then we can compute a function
  dominating $g$, and by
  Lemma~\ref{lem:domination-gives-transitivity}, compute an infinite
  $g$-transitive set, contradicting our assumption.
\end{proof}

By the previous lemma, $G$ is not of DNC degree for every sufficiently
generic for this notion of forcing.  This completes the proof.
\end{proof}

\begin{definition}[Ascending descending sequence]
  Given a linear order~$(L, <_L)$, an \emph{ascending}
  (\emph{descending}) sequence is a set~$S$ such that for
  every~$x <_\Nb y \in S$, $x <_L y$ ($x >_L y$).  $\ads$ is the
  statement ``Every infinite linear order admits an infinite ascending
  or descending sequence''.  $\sads$ is the restriction of $\ads$ to
  orders of type $\omega^{*} + \omega$.
\end{definition}

\begin{corollary}
  $\rca + \dgap \nvdash \sads$.
\end{corollary}
\begin{proof}
  By Tennenbaum (see Rosenstein~\cite{Rosenstein1982Linear}), there is
  a computable linear ordering $\Lcal$ of order type
  $\omega+\omega^{*}$ with no computable infinite ascending or
  descending sequence. Let $U$ and $V$ the the $\omega$ and the
  $\omega^{*}$ part, respectively. In particular, both $U$ and $V$
  must be hyperimmune. By a relativization of
  Theorem~\ref{thm:gap-preserve-hyperimmune-and-non-dnc}, there is a
  Turing ideal $\Mcal \models \dgap$ such that $U$ and $V$ are
  hyperimmune relative to every element of this model.  However, $U$
  is not hyperimmune relative to any infinite ascending sequence for
  $\Lcal$, and the $V$ part is not hyperimmune relative to any
  infinite descending sequence for $\Lcal$.  Therefore,
  $\Mcal \not \models \sads$.
\end{proof}

\begin{corollary}
  $\rca + \dgap \nvdash \dnc$.
\end{corollary}
\begin{proof}
  By a relativized version of
  Theorem~\ref{thm:gap-preserve-hyperimmune-and-non-dnc}, there is a
  Turing ideal $\Mcal \models \dgap$ with no DNC function.  In
  particular, $\Mcal \not \models \dnc$.
\end{proof}

% Recall that given an infinite set $H = \{x_0 < x_1 < \dots \}$, the
% \emph{principal function} $p_H$ is defined by $p_H(n) = x_n$.

\begin{theorem}\label{thm:dgap-escapes-delta2}
  For every $\Delta^0_2$ function $f$, there is a $\Delta^0_2$
  increasing function $g$ such that $f$ does not dominate $p_H$ for
  any infinite $g$-transitive set $H$.
\end{theorem}
\begin{proof}
  Let $g : \omega \to \omega$ be the $\Delta^0_2$ function which on
  input $x$, returns $\max \{ f(y) : y \leq x + 1 \}$. Let
  $H = \{x_0 < x_1 < \dots\}$ be an infinite $g$-transitive set. We
  claim that $f$ does not dominate $p_H$. Since $H$ is $g$-transitive,
  there must be some $n$ such that $[x_n, x_{n+1}]$ is $g$-large,
  otherwise $[x_i, x_j]$ would be $g$-small for every $i < j$. In
  particular, $p_H(n+1) = x_{n+1} \geq g(x_n) \geq f(n+1)$ since
  $n \leq x_n$. This completes the proof.
\end{proof}

\begin{corollary}
  Every $\omega$-model of $\dgap$ is a model of $\amt$.
\end{corollary}
\begin{proof}
  Csima et al.~\cite{Csima2004Bounding} and
  Conidis~\cite{Conidis2008Classifying} proved that $\amt$ is
  computably equivalent to the statement \qt{For every $\Delta^0_2$
    function $f : \omega \to \omega$, there is a function not
    dominated by $f$.} Apply Theorem~\ref{thm:dgap-escapes-delta2}.
\end{proof}

\begin{theorem}\label{thm:dgap-low-solutions}
  For every $\Delta^0_2$ increasing function $g$, there is an infinite
  $g$-transitive set $H$ of low degree.
\end{theorem}
\begin{proof}
  Suppose there is no computable infinite $g$-transitive set,
  otherwise we are done.  We construct the set $H$ by the finite
  extension method. For this, we define a $\Delta^0_2$ sequence
  $F_0 \subsetneq F_1 \subsetneq \dots$ of finite $g$-transitive sets
  such that $F_{n+1} \setminus F_n > F_n$, and then set
  $H = \bigcup_n F_n$.  Start with $F_0 = \{0\}$.  At stage $e$,
  search $\emptyset'$-computably for one of the following:
  \begin{itemize}
  \item[(i)] Some non-empty finite set $E > F_e$ such that
    $F_e \cup E$ is $g$-transitive, and
    $\Phi_e^{F_e \cup E}(e)\downarrow$
  \item[(ii)] Some $n$ such that for every non-empty finite set
    $E > n$, $\Phi_e^{F_e \cup E}(e)\uparrow$.
  \end{itemize}
  We claim that one of the two must be found. Indeed, suppose (ii) is
  not found. Then there is an infinite computable sequence of
  non-empty finite sets $E_0 < E_1 < \dots$ such that
  $\Phi_e^{F_e \cup E_i}(e)\downarrow$.  If $E_i$ is $g$-large
  (meaning that $[\min E_i, \max E_i]$ is $g$-large) for all but
  finitely many $i$, then we can computably dominate the function
  $g$. Any function $f$ dominating $g$ computes a $g$-transitive set
  by constructing an infinite set consisting only of $f$-large (hence
  $g$-large) intervals. This contradicts the assumption that there is
  no computable infinite $g$-transitive set. If $E_i$ is $g$-small for
  infinitely many $i$, then pick some $i$ such that
  $[\max F_e, \min E_i]$ is $g$-large. Then $F_e \cup E_i$ is
  $g$-transitive, and we are in the case (i).

  If we are in case (i), set $F_{e+1} = F_e \cup E$, and if we are in
  case (ii), pick some $x > n$ such that $[\max F_e, x]$ is $g$-large,
  and set $F_{e+1} = F_e \cup \{x\}$. This completes the construction.
\end{proof}

\begin{corollary}
  $\rca + \dgap \nvdash \semo$.
\end{corollary}
\begin{proof}
  Using a relativized version of Theorem~\ref{thm:dgap-low-solutions},
  we can build a Turing ideal $\Mcal \models \dgap$ with only low
  sets.  By Kreuzer~\cite{Kreuzer2012Primitive},
  $\Mcal \not \models \semo$.
\end{proof}

% \ludovic{TODO, my guess is that $\rt^2_2$ does not imply
% $\dgap$. Prove it.}

\begin{definition}
  An infinite set $H$ is \emph{immune} if it has no computable
  infinite subset.  A infinite set $H$ is \emph{$k$-immune} if there
  is no computable sequence of non-empty sets $F_0, F_1, \dots$ such
  that $F_i > i$, $F_i$ has at most $k$ elements, and
  $F_i \cap H \neq \emptyset$.  An infinite set $H$ is
  \emph{constant-bound immune} if it is $k$-immune for every $k$.
\end{definition}

\begin{definition}
  A function $g$ is \emph{left-c.e.} if $\{(x, v) : g(x) \geq v\}$ is
  c.e.
\end{definition}

We are actually interested in the case when the function $g$ is a
modulus for the halting set. In particular, $\emptyset'$ admits a
left-c.e.\ modulus $\mu_{\emptyset'}$. With the extra assumption that
the function $g$ is left-c.e., we can obtain a stronger preservation
property, namely, preservation of countably many immune sets. The
second theorem can be used to construct models of $\cac$ together with
the statement \qt{For every $X$, there is an infinite transitive set
  for the modulus of $X'$\ } which is not a model of $\dnc$, since
$\cac$ admits preservation of one constant-bound immunity which $\dnc$
does not (see Patey~\cite{Patey2016Partial}).

\begin{theorem}
  Let $g$ be an increasing left-c.e.\ function, $B_0, B_1, \dots$ be a
  countable sequence of immune sets.  Then there is an infinite
  $g$-transitive set $H$ such that $B_i$ is $H$-immune for every $i$.
\end{theorem}
\begin{proof}
  We construct the set $H$ by the finite extension method. For this,
  we define a sequence $F_0 \subsetneq F_1 \subsetneq \dots$ of finite
  $g$-transitive sets such that $F_{n+1} \setminus F_n > F_n$, and
  then set $H = \bigcup_n F_n$. At stage $s = \langle e, i\rangle$, we
  want to satisfy the following requirement:
$$
\Rcal_{e,i} \mbox{ : } W_e^H \mbox{ is finite or } W_e^H \not
\subseteq B_i
$$
Assume $F_s$ is already defined, and fix some threshold $t$ such that
$[\max F_s, t]$ is $g$-large.  For every $n$, computably search for a
$g$-small finite set $E_n \geq \max(t, n)$ such that
$W_e^{F_s \cup E_n} \cap (n, \infty) \neq \emptyset$. If not found,
$E_n$ is undefined. Note that this search can be made computably since
the function $g$ is left-c.e., so being $g$-small is a c.e.\ event.
If there is some $n$ such that $E_n$ is undefined, then set
$F_{s+1} = F_s \cup \{t\}$. We have ensured that $\max W_e^H \leq n$.
If $E_n$ is defined for every $n$, then there must be some $n$ such
that $W_e^{F_s \cup E_n} \not \subseteq B_i$, otherwise
$\bigcup_n W_e^{F_s \cup E_n}$ would be an infinite c.e.\ subset of
$B_i$, contradicting immunity of $B_i$. Then letting
$F_{s+1} = F_s \cup E_n$, we ensured that $W_e^H \not \subseteq
B_i$. We then go to the next stage.  This completes the construction.
\end{proof}

\begin{theorem}
  Let $g$ be an increasing $\Delta^0_2$ function dominating
  $\mu_{\emptyset'}$, and let $B_0, B_1, \dots$ be a countable
  sequence of constant-bound immune sets.  Then there is an infinite
  $g$-transitive set $G$ such that $B_i$ is constant-bound $G$-immune
  for every $i$.
\end{theorem}
\begin{proof}
  We construct the set $G$ by the finite extension method. For this,
  we define a sequence $F_0 \subsetneq F_1 \subsetneq \dots$ of finite
  $g$-transitive sets such that $F_{n+1} \setminus F_n > F_n$, and
  then set $G = \bigcup_n F_n$. At stage $s = \langle e, k, i\rangle$,
  we want to satisfy the following requirement:
$$
\Rcal_{e,k,i} \mbox{ : } \Phi_{e,k}^G \mbox{ is finite or } (\exists
n)\Phi_e{e,k}^G(n) \cap B_i = \emptyset
$$
where $\Phi_{0,k}, \Phi_{1,k}, \dots$ is an effective enumeration of
all $k$-enumeration functionals, that is, whenever
$\Phi^G_{e,k}(n) \downarrow$, then $\Phi^G_{e,k}(n)$ is interpreted as
a $k$-set of elements greater than $n$.  Assume $F_s$ is already
defined, and fix some threshold $t$ such that $[\max F_s, t]$ is
$g$-large.

Define a partial computable $2k$-enumeration $U_0, U_1, \dots$ as
follows.  For every $n$, computably search for two finite sets
$E, H \geq \max(t, n)$ such that
$\Phi_{e,k}^{F_s \cup E}(n)\downarrow$,
$\Phi_{e,k}^{F_s \cup H}(n)\downarrow$, and $H$ is $g$-transitive
according to $\emptyset'_{\max E}\uh \min E$, where $\emptyset'_t$ is
the approximation of $\emptyset'$ at stage $t$. If not found, $U_n$ is
undefined. Otherwise,
$U_n = \Phi_{e,k}^{F_s \cup E}(n) \cup \Phi_{e,k}^{F_s \cup
  H}(n)$. Note that since $g$ is $\Delta^0_2$, $\emptyset'$ can decide
whether a set is $g$-small or not. We have two cases.

Case 1: there is some $n$ such that $U_n$ is undefined. Suppose first
there is some $g$-transitive finite set $E \geq \max(t, n)$ such that
$\Phi_{e,k}^{F_s \cup E}(n)\downarrow$. Let $u$ be sufficiently large
so that $E$ is $g$-transitive according to $\emptyset'_u \uh
u$. Letting $F_{s+1} = F_s \cup \{u\}$, we have forced
$\Phi^G_{e,k}(n)\uparrow$. Suppose now that whenever
$E \geq \max(t, n)$ is $g$-transitive, then
$\Phi_{e,k}^{F_s \cup E}(n)\uparrow$. Then letting
$F_{s+1} = F_s \cup \{t\}$, we have forced $\Phi^G_{e,k}(n)\uparrow$.

Case 2: $U_n$ is defined for every $n$. Then by constant-bound
immunity of $B_i$, there must be some $n$ such that
$U_n \cap B_i = \emptyset$. Let $E$ and $H$ witness that $U_n$ is
defined. Suppose first that $E$ is $g$-small. Then letting
$F_{s+1} = F_s \cup E$, we have forced
$\Phi_{e,k}^G(n)\downarrow \cap B_i = \emptyset$. Suppose now that $E$
is $g$-large. Since $g$ dominates $\mu_{\emptyset'}$,
$\emptyset'_{\max E}\uh \min E$ agrees with $\emptyset'$ and $H$ is
truly $g$-transitive. Then letting $F_{s+1} = F_s \cup H$, we have
forced $\Phi_{e,k}^G(n)\downarrow \cap B_i = \emptyset$. We then go to
the next stage.  This completes the construction.
\end{proof}

% \section{The thin set theorems under
% reducibilities}\label{sect:reducibilities}
% \input{parts/part5-reducibilities}

% \appendix

% \section{Inductive proof of strong cone avoidance of $\rt^2_{3, 2}$}
% \input{parts/part7-rt23-strong-cone-avoidance}

% \section{Direct proof of strong cone avoidance of $\rt^2_{3, 2}$}
% \input{parts/part8-rt23-strong-cone-avoidance-direct}

% \section{Direct proof of strong cone avoidance of $\rt^3_{5, 4}$}
% \input{parts/part9-rt35-strong-cone-avoidance-direct}

% \section{Direct proof of strong cone avoidance of
% $\rt^n_{k, 2^{n-1}}$}
% \input{parts/part10-rtn-strong-cone-avoidance-direct}

% \section{Direct proof of strong cone avoidance of
% $\rt^n_{<\infty, \ell}$}
% \input{parts/part11-grtn-strong-cone-avoidance-direct}

\bibliographystyle{plain} %\bibliography{bibliography.bib}

\bibliography{thin.bib}

\end{document}